\documentclass[10pt,reqno]{amsart}

\usepackage[dvipsnames]{xcolor}
\usepackage{amsmath,amssymb,amsthm,amsfonts,enumerate,tikz,bm}
\usepackage{mathrsfs}
\usetikzlibrary{matrix,arrows,positioning,calc}
\usepackage[all]{xy}
\usepackage[bookmarksnumbered,colorlinks]{hyperref}
\usepackage{url}
\numberwithin{equation}{section}
\usepackage{graphicx}
\usepackage[T1]{fontenc}
\usepackage{textcomp}
\usepackage{palatino,helvet}

\usepackage{footmisc}

\newtheorem{definition}{Definition}[section]
\newtheorem{remark}[definition]{Remark}
\newtheorem{example}[definition]{Example}

\newtheorem{proposition}[definition]{Proposition}
\newtheorem{lemma}[definition]{Lemma}
\newtheorem{corollary}[definition]{Corollary}

\theoremstyle{remark}

\newcommand{\Ext}{\mathrm{Ext}}
\newcommand{\Tor}{\mathrm{Tor}}

\newcommand{\Hom}{\mathrm{Hom}}

\newcommand{\A}{\mathcal{A}}
\newcommand{\B}{\mathcal{B}}
\newcommand{\M}{\mathcal{M}}

\newcommand{\He}{\mathcal{H}}

\newcommand{\X}{\mathcal{X}}
\newcommand{\Y}{\mathcal{Y}}
\newcommand{\Z}{\mathcal{Z}}

\newcommand{\pd}{\mathrm{pd}}
\newcommand{\op}{\mathrm{op}}

\newcommand{\Proj}{\mathcal{P}}

\newcommand{\Inj}{\mathcal{I}}

\newcommand{\id}{\mathrm{id}}

\newcommand{\resdim}{\mathrm{resdim}}

\newcommand{\coresdim}{\mathrm{coresdim}}

\newcommand{\FP}{\mathsf{FP}}

\newcommand{\Modu}{\mathsf{Mod}}
\newcommand{\Ker}{\mathrm{Ker}}

\newcommand{\Flat}{\mathcal{F}}

\newcommand{\GF}{\mathcal{GF}}

\newcommand{\GI}{\mathcal{GI}}
\newcommand{\N}{\mathcal{N}}
\newcommand{\Coker}{\mathrm{CoKer}}

\usepackage{latexsym,amssymb,amscd} 
\usepackage{amsmath}
\usepackage[all]{xypic}
\usepackage{amsthm}

\newcommand{\ortogonal}{\bot}

\begin{document}
\title[Gorensteinness from duality pairs induced via Foxby equivalences]{Gorensteinness from duality pairs \\ induced via Foxby equivalences}

\author{V\'ictor Becerril}
\address[V. Becerril]{Centro de Ciencias Matem\'aticas. Universidad Nacional Aut\'onoma de M\'exico. CP58089. Morelia, M\'exico}
\email{victorbecerril@matmor.unam.mx}

\author{Marco A. P\'erez}
\address[M. A. P\'erez]{Instituto de Matem\'atica y Estad\'istica ``Prof. Ing. Rafael Laguardia''. Facultad de Ingenier\'ia. Universidad de la Rep\'ublica. CP11300 Montevideo, Uruguay}
\email{mperez@fing.edu.uy}


\begin{abstract} 
We define and study induced duality pairs under Foxby equivalences. Given a semidualizing $(S,R)$-bimodule ${}_S C_R$, if 
\[
(\mathcal{A}_C(R),\mathcal{B}_C(R^{\rm op})) \ \ \text{ and } \ \ (\mathcal{A}_C(S^{\rm op}),\mathcal{B}_C(S))
\]
denote the duality pairs formed by the corresponding classes of Auslander and Bass modules, and if $(\mathcal{M,N})$ is a duality pair over $R$, we study the duality pair formed by the essential images of the restricted Foxby equivalences
\[
(C \otimes_R \sim)|_{\mathcal{A}_C(R) \cap \mathcal{M}} \ \ \text{ and } \ \ \Hom_{R^{\rm op}}(C,\sim) |_{\mathcal{B}_C(R^{\rm op}) \cap \mathcal{N}},
\] 
denoted by $\mathcal{M}^C(S)$ and $\mathcal{N}^C(S^{\rm op})$. We investigate which additional properties of the duality pair $(\mathcal{M,N})$ are transferred to $(\mathcal{M}^C(S),\mathcal{N}^C(S^{\rm op}))$. We also study several versions of Gorenstein injective and Gorenstein flat modules relative to the pairs $(\mathcal{A}_C(R) \cap \mathcal{M},\mathcal{B}_C(R^{\rm op}) \cap \mathcal{N})$ and $(\mathcal{M}^C(S),\mathcal{N}^C(S^{\rm op}))$. For instance, we explore the relation between these classes of modules under Foxby equivalences and under Pontryagin duality.  
\end{abstract}  
\maketitle


\setcounter{tocdepth}{1}
\tableofcontents


\section*{Introduction} 

Duality pairs were introduced in 2009 by Holm and J{\o}rgensen in \cite{HJ09}. Under certain conditions, these pairs are a source of left and right approximations with respect to the classes forming the pairs. For instance, given a duality pair $(\mathcal{M,N})$ over a ring $R$ (associative with identity) with $\mathcal{M}$ closed under (co)products, then every (left) $R$-module has an $\mathcal{M}$-preenvelope (resp., a $\mathcal{M}$-cover). Duality pairs are then important in relative homological algebra. On the other hand, the main example of duality pair considered in this article is the one formed by the classes of Auslander and Bass modules associated to a semidualizing $(S,R)$-bimodule ${}_S C_R$. These modules were introduced in 1973 by Foxby \cite{Foxby73}. If $\mathcal{A}_C(R)$ and $\mathcal{B}_C(S)$ denote the classes of Auslander and Bass modules over $R$ and $S$, then there is an equivalence of categories, known as the first Foxby equivalence, given by the functor $C \otimes_R \sim \colon \mathcal{A}_C(R) \to \mathcal{B}_C(S)$ whose quasi-inverse is given by $\Hom_{S}(C,\sim)$. Similarly, one has the second Foxby equivalence $- \otimes_S C \colon \mathcal{A}_C(S^{\rm op}) \to \mathcal{B}_C(R^{\rm op})$ whose quasi-inverse is given by $\Hom_{R^{\rm op}}(C,\sim)$. Besides these equivalences, another relation between Auslander and Bass modules is given by the Pontryagin duality. Specifically, it was proved recently by Huang \cite{Huang21} that $(\mathcal{A}_C(R),\mathcal{B}_C(R^{\rm op}))$ and $(\mathcal{A}_C(S^{\rm op}),\mathcal{B}_C(S))$ are complete duality pairs over $R$ and $S^{\rm op}$, respectively. 

The first main goal of the present article is to induce duality pairs over $S$ from a duality $(\mathcal{M,N})$ over $R$ via Foxby equivalences. Motivated by the work of Wu and Gao \cite{WuGao22} on $\text{FP}_n$-injective and $\text{FP}_n$-flat modules relative to a semidualizing bimodule, the idea is to consider the intersection pair $(\mathcal{A}_C(R) \cap \mathcal{M}, \mathcal{B}_C(R^{\rm op}) \cap \mathcal{N})$, also a duality pair, and to define the pair formed by the essential images of the restricted equivalences $(C \otimes_R \sim)|_{\mathcal{A}_C(R) \cap \mathcal{M}}$ and $\Hom_{R^{\rm op}}(C,\sim) |_{\mathcal{B}_C(R^{\rm op}) \cap \mathcal{N}}$, denoted by $\mathcal{M}^C(S)$ and $\mathcal{N}^C(S^{\rm op})$. It turns out that $(\mathcal{M}^C(S),\mathcal{N}^C(S^{\rm op}))$ is a duality pair over $S$, and that many of the additional properties that the duality pair $(\mathcal{M,N})$ may have are inherited by $(\mathcal{M}^C(S),\mathcal{N}^C(S^{\rm op}))$. For instance, we show that if $(\mathcal{M,N})$ is symmetric or perfect, then so is  $(\mathcal{M}^C(S),\mathcal{N}^C(S^{\rm op}))$. 

Our second main goal has to do with Gorenstein homological algebra relative to the pairs $(\mathcal{A}_C(R) \cap \mathcal{M}, \mathcal{B}_C(R^{\rm op}) \cap \mathcal{N})$ and $(\mathcal{M}^C(S),\mathcal{N}^C(S^{\rm op}))$. We study Gorenstein injective and Gorenstein flat modules relative to these pairs, and show several relations between these modules under Foxby equivalences and Pontryagin duality, and exhibit some homotopical aspects regarding model structures and their homotopy categories. These versions of relative Gorenstein modules are based on recent previous works of Wang, Yang and Zhu \cite{WangYangZhu19}, and of the first and second authors \cite{BMS,BecerrilPerez25}.  

In general, our findings cover several results of Wu and Gao's \cite{WuGao22} and of Cheng and Zhao's \cite{ChengZhao24}. This article provides a general theory of induced duality pairs in Foxby equivalences and  Gorenstein homological algebra relative to these pairs, generalizing some aspects of the theory of $C$-$\text{FP}_n$-flat and $C$-$\text{FP}_n$-injective modules developed in \cite{WuGao22}, and of the theory of Gorenstein $G_C$-$\text{FP}_n$-injective modules from \cite{ChengZhao24}.


\subsection*{Organization of this article} 

In Section \ref{sec:prelims} we recall the basic terminology, preliminary results and examples necessary for the rest of the article. 

Section \ref{sec:induced} is devoted to study the properties of the pair $(\mathcal{M}^C(S),\mathcal{N}^C(S^{\rm op}))$. We first provide in Lemma \ref{Lema1} characterizations of $\mathcal{M}^C(S)$ and $\mathcal{N}^C(S^{\rm op})$, which allow us to note in Proposition \ref{Equivalencia1} that there is an equivalence of categories between $\mathcal{A}_C(R) \cap \mathcal{M}$ and $\mathcal{M}^C(S)$, and between $\mathcal{B}_C(R^{\rm op}) \cap \mathcal{N}$ and $\mathcal{N}^C(S^{\rm op})$. In Proposition \ref{Dualidad1} we show that if $(\mathcal{M,N})$ is a (symmetric) duality pair, then so is $(\mathcal{M}^C(S),\mathcal{N}^C(S^{\rm op}))$. The rest of the properties that can be transferred from $(\mathcal{M,N})$ to $(\mathcal{M}^C(S),\mathcal{N}^C(S^{\rm op}))$ are described in Proposition \ref{Prop0}. We point out in particular that if $(\M,\N)$ is complete and $S \in \M^C(S)$, then so is $(\M^C(S),\N^C(S^{\rm op}))$. For the rest of Section \ref{sec:induced}, relative homological dimensions are considered with respect to the classes $\mathcal{M}^C(S)$ and $\mathcal{N}^C(S^{\rm op})$. In the case where ${}_S C_R$ is faithfully semidualizing, we extend the equivalence of categories in Proposition \ref{Equivalencia1} to the subcategories of modules with bounded $\mathcal{M}^C(S)$-resolution dimension and bounded $\mathcal{N}^C(S^{\rm op})$-coresolution dimension (see Proposition \ref{extended-equivalence}).  

Section \ref{sec:relGFGI} is formed by results on Gorenstein homological and homotopical algebra relative to $(\mathcal{A}_C(R) \cap \mathcal{M}, \mathcal{B}_C(R^{\rm op}) \cap \mathcal{N})$ and $(\mathcal{M}^C(S),\mathcal{N}^C(S^{\rm op}))$. Cores of these pairs will be key in this section, namely, $\mathcal{H}(\mathfrak{M}) = \mathfrak{M} \cap \mathfrak{M}^\perp$ and $\mathcal{H}(\mathfrak{N}) = {}^\perp\mathfrak{N} \cap \mathfrak{N}$, where $\mathfrak{M} = \mathcal{A}_C(R) \cap \mathcal{M}$ and $\mathfrak{N} = \mathcal{B}_C(R^{\rm op}) \cap \mathcal{N}$. We also consider the essential images of the equivalences $(C \otimes_R \sim)|_{\mathcal{H}(\mathfrak{M})}$ and $\Hom_{R^{\rm op}}(C,\sim) |_{\mathcal{H}(\mathfrak{N})}$, denoted by $\mathcal{H}^C(\mathfrak{M})$ and $\mathcal{H}_C(\mathfrak{N})$, respectively. We start Section \ref{sec:relGFGI} studying $\mathcal{N}^C(S^{\rm op})$-Gorenstein injective $S^{\rm op}$-modules and  $\mathfrak{N}$-Gorenstein injective $R^{\rm op}$-modules. Concretely, Proposition \ref{duality_Ginj} and Corollary \ref{Gorenstein-injective-equivalence} show how these relative Gorenstein injective modules behave under Foxby equivalences. The counterparts of these modules are given by the $(\M^C(S), \N^C(S^{\rm op}))$-Gorenstein flat $S$-modules and the $(\mathfrak{M,N})$-Gorenstein flat $R$-modules. Proposition \ref{GF-Equiv} and Corollary \ref{Gorenstein-flat-equivalence} show the relation between these modules under Foxby equivalences. Moreover, the Pontryagin duality relation between these classes of relative Gorenstein flat and Gorenstein injective modules is provided in Proposition \ref{Dual-P}. We follow a similar approach for the $(\mathcal{H}_C(\mathfrak{N}),\mathcal{N}^C(S^{\rm op}))$-Gorenstein injective $S^{\rm op}$-modules, the $(\mathcal{H}(\mathfrak{N}),\mathfrak{N})$-Gorenstein injective $R^{\rm op}$-modules, the $(\mathcal{M}^C(S),\mathcal{H}_C(\mathfrak{N}))$-Gorenstein flat $S$-modules and the $(\mathfrak{M},\mathcal{H}(\mathfrak{N}))$-Gorenstein flat $R$-modules (see Propositions \ref{prop:Foxby-corazon} and \ref{prop:dualityR}, \ref{otra_dualidad}). We shall also study in the last part of Section \ref{sec:relGFGI} ``weak'' versions of these relative Gorenstein injective and Gorenstein flat modules, and relate them to their previous (``strong'') analogs. Finally, regarding homotopical aspects, we comment in Corollary \ref{coro:first_model_structure} and Proposition \ref{GFmodelST} on sufficient conditions for the existence of two abelian model structures, and their homotopy categories, where $\GF_{(\M^C(S),\N^C(S^{\rm op}))}$ (resp., $\M^C(S)$) and $\GF_{(\mathfrak{M},\mathcal{H}(\mathfrak{N}))}$ (resp., $\mathfrak{M}$) are the classes of (trivially) cofibrant objects.


\section{Preliminaries} \label{sec:prelims}


\subsection*{Notation for modules and chain complexes}

In what follows, we mainly work within the category $\Modu(R)$ of (unitary) (left) $R$-modules, where $R$ is an associative ring with identity. Right $R$-modules are modules over the opposite ring $R^{\rm op}$. In some cases, in order to indicate whether $M \in \Modu(R)$ or $M \in \Modu(R^{\rm op})$, we write ${}_R M$ or $M_R$, respectively. As special classes of objects in $\Modu(R)$, we denote by $\Proj(R)$, $\Inj(R)$ and $\Flat(R)$ the classes of projective, injective and flat $R$-modules, respectively.


\subsection*{Orthogonality}

Concerning functors defined on modules, $\Ext^i_R(-,\sim)$ denotes the right $i$-th derived functor of
\[
\Hom_R(-,\sim) \colon [\Modu(R)]^{\rm op} \times \Modu(R) \longrightarrow \Modu(\mathbb{Z}).
\]
On the other hand, $\Tor^R_i(-,\sim)$ denotes the left derived functor of 
\[
- \otimes_R \sim \colon \Modu(R^{\rm op}) \times \Modu(R) \longrightarrow \Modu(\mathbb{Z}).
\]  
Let $\mathcal{X} \subseteq \Modu(R)$, $N \in \Modu(R)$ and $i \in \mathbb{Z}_{>0}$. The expression $\Ext^i_R(\mathcal{X},N) = 0$ means that $\Ext^i_R(X,N) = 0$ for every $X \in \mathcal{X}$. Moreover, $\Ext^i_R(\mathcal{X,Y}) = 0$ if $\Ext^i_R(\mathcal{X},Y) = 0$ for every $Y \in \mathcal{Y}$. The expression $\Ext^i_R(N,\mathcal{Y}) = 0$ has a similar meaning. Moreover, by $\Ext^{\geq 1}_R(M,N) = 0$ we mean that $\Ext^i_R(M,N) = 0$ for every $i \geq 1$. One also has similar meanings for $\Ext^{\geq 1}_R(\mathcal{X},N) = 0$, $\Ext^{\geq 1}_R(N,\mathcal{Y}) = 0$ and $\Ext^{\geq 1}_R(\mathcal{X,Y}) = 0$. We can also replace $\Ext$ by $\Tor$ in order to obtain the corresponding notations for $\Tor$-orthogonality. The \emph{$i$-th and total left Ext-orthogonal and Tor-orthogonal complements} of $\mathcal{X}$ will be denoted by
\[
{}^{\perp_i}\mathcal{X} := \{ M \in \Modu(R) {\rm \ : \ } \Ext^i_{R}(M,\X) = 0 \} \ \ \ \text{and} \ \ \ {}^{\perp}\X := \bigcap_{i \geq 1} {}^{\perp_i}\X,
\]
\[
{}^{\top_i}\mathcal{X} := \{ N \in \Modu(R^{\rm op}) {\rm \ : \ } \Tor^{R}_i(N,\X) = 0 \} \ \ \ \text{and} \ \ \ {}^\top\X := \bigcap_{i \geq 1} {}^{\top_i}\mathcal{X}.
\]
The \emph{$i$-th and total right Ext-orthogonal and Tor-orthogonal complements} are defined similarly.


\subsection*{Relative homological dimensions} 

Let $M \in  \Modu (R)$ and $\mathcal{X}, \mathcal{Y} \subseteq  \Modu(R)$. The \emph{injective dimensions of $M$ and $\mathcal{Y}$ relative to $\mathcal{X}$} are defined by
\begin{align*}
\id_{\mathcal{X}}(M) & := \inf \{ m \in \mathbb{Z}_{\geq 0} \text{ : } \Ext _R ^{\geq m+1}(\mathcal{X},M) = 0  \}, \\ \id_{\mathcal{X}}(\mathcal{Y}) & := \sup \{ \id_{\mathcal{X}}(Y) \text{ : } Y \in \mathcal{Y} \}.
\end{align*}
If $\mathcal{X} =  \Modu(R)$, we write $\id_{ \Modu(R)}(M) = \id(M)$ and $\id_{\mathsf{Mod}(R)}(\mathcal{Y}) = \id(\mathcal{Y})$ for the (absolute) injective dimensions of $M$ and $\mathcal{Y}$.  Dually we can define the relative and absolute projective dimensions $\pd_{\X} (M)$, $\pd_{\X}(\Y)$, $\pd(M)$ and $\pd(\Y)$.

By an \emph{$\mathcal{X}$-resolution of $M$} we mean an exact complex 
\[
\cdots \to X_m \to X_{m-1} \to \cdots \to X_1 \to X_0 \to M \to 0
\]
with $X_k \in \mathcal{X}$ for every $k \in \mathbb{Z}_{\geq 0}$.  If $X_k = 0$ for $k > m$ and $X_k \neq 0$ for every $0 \leq k \leq m$, we say that the previous resolution has \emph{length} $m$. The \emph{resolution dimension relative to $\mathcal{X}$} (or the \emph{$\mathcal{X}$-resolution dimension}) of $M$ is defined as the value
\[
\resdim_{\mathcal{X}}(M) := \min \{ m \in \mathbb{Z}_{\geq 0} \ \mbox{ : } \ \text{there exists an $\mathcal{X}$-resolution of $M$ of length $m$} \}.
\]
If $\mathcal{X}$ is \emph{pointed} (that is, $0 \in \mathcal{X}$) and closed under isomorphisms, we set $\resdim_{\mathcal{X}}(M) = 0$ whenever $M \simeq 0$. Moreover, if $\mathcal{Y} \subseteq  \Modu (R)$ then
\[
\resdim_{\mathcal{X}}(\mathcal{Y}) := \sup \{ \resdim_{\mathcal{X}}(Y) \ \mbox{ : } \ \text{$Y \in \mathcal{Y}$} \}
\]
defines the \emph{resolution dimension of $\mathcal{Y}$ relative to $\mathcal{X}$}. The classes of objects with bounded (by some $n \geq 0$) and finite $\mathcal{X}$-resolution dimensions will be denoted by
\[
\mathcal{X}^\wedge_n := \{ M \in  \Modu(R) \text{ : } \resdim_{\mathcal{X}}(M) \leq n \} \ \ \ \text{and} \ \ \ \mathcal{X}^\wedge := \bigcup_{n \geq 0} \mathcal{X}^\wedge_n.
\]
Dually, we can define \emph{$\mathcal{X}$-coresolutions} and the \emph{coresolution dimension of $M$ and $\mathcal{Y}$ relative to $\mathcal{X}$} (denoted $\coresdim_{\mathcal{X}}(M)$ and $\coresdim_{\mathcal{X}}(\mathcal{Y})$). We also have the dual notations $\mathcal{X}^\vee_n$ and $\mathcal{X}^\vee$ for the classes of $R$-modules with bounded and finite $\mathcal{X}$-coresolution dimension.


\subsection*{Approximations}

Given a class $\mathcal{X}$ of $R$-modules and $M \in \Modu(R)$, recall that a morphism $\varphi \colon X \to M$ with $X \in \mathcal{X}$ is an \emph{$\mathcal{X}$-precover of $M$} if for every morphism $\varphi' \colon X' \to M$ with $X' \in \mathcal{X}$, there exists a morphism $h \colon X' \to X$ such that $\varphi' = \varphi \circ h$. An $\X$-precover $\varphi$ of $M$ is:
\begin{itemize}
\item an \emph{$\X$-cover} if in the case $X' = X$ and $\varphi' = \varphi$, the equation $\varphi = \varphi \circ h$ can only be completed by automorphisms $h \colon X \to X$ of $X$;
\item a \emph{special $\X$-cover} if $\Coker (\varphi) = 0$ and $\Ker (\varphi) \in \X^{\ortogonal _1}$.
\end{itemize}

 A class $\mathcal{X} \subseteq \Modu(R)$ is \emph{precovering} if every $R$-module has an $\mathcal{X}$-precover. \emph{Covering} and \emph{special precovering} classes are defined similarly. Dually, one has the notions of \emph{envelopes}, \emph{preenvelopes} and \emph{special preenvelopes} of a module relative to $\X$, and of (\emph{special}) \emph{preenveloping} and \emph{enveloping} classes.


\subsection*{Cotorsion pairs}

Two classes $\mathcal{X,Y} \subseteq \Modu (R)$ of $R$-modules form a \emph{cotorsion pair} $(\mathcal{X,Y})$ if $\mathcal{X} = {}^{\perp_1}\mathcal{Y}$ and $\mathcal{Y} = \mathcal{X}^{\perp_1}$. A cotorsion pair $(\mathcal{X,Y})$ is:
\begin{itemize}
\item \emph{complete} if $\mathcal{X}$ is special precovering (or equivalently, if $\mathcal{Y}$ is special preenveloping); 

\item \emph{hereditary} if $\Ext^{\geq 1}_R(\mathcal{X,Y}) = 0$ (or equivalently, if $\mathcal{X}$ is resolving or $\mathcal{Y}$ is coresolving). 
\end{itemize}
By \emph{resolving} we mean that $\mathcal{X}$ is closed under extensions and kernels of epimorphisms, and $\Proj(R) \subseteq \X$. \emph{Coresolving} classes are defined dually. Note that $(\mathcal{X,Y})$ is a hereditary cotorsion pair if, and only if, $\mathcal{X} = {}^{\perp}\mathcal{Y}$ and $\mathcal{Y} = \mathcal{X}^\perp$. Finally, recall that a cotorsion pair $(\X,\Y)$ is cogenerated by a set $\mathcal{S} \subseteq \X$ if $\mathcal{Y} = \mathcal{S}^{\perp_1}$. A well known result by Eklof and Trlifaj \cite{EklofTrlifaj01} shows that such pairs are complete.


\subsection*{Duality pairs}

The notion of duality pair was introduced by Holm and J{\o}rgensen in \cite{HJ09}. Two classes $\mathcal{M} \subseteq \Modu(R)$ and $\mathcal{N} \subseteq \Modu (R^{\rm op})$ form a \emph{duality pair} $(\mathcal{M,N})$ over $R$ if:
\begin{enumerate}
\item $M \in \mathcal{M}$ if, and only if, $M^+ \in \mathcal{N}$, where $M^+ := \Hom_{\mathbb{Z}}(M,\mathbb{Q/Z})$ denotes the \emph{character} or \emph{Pontryagin dual} of $M$.  

\item $\mathcal{N}$ is closed under direct summands and finite direct sums. 
\end{enumerate}
One has a similar notion of duality pairs over $R^{\rm op}$ in the case where $\mathcal{M} \subseteq \Modu(R^{\rm op})$ and $\mathcal{N} \subseteq \Modu(R)$, since $(-)^+ = \Hom_{\mathbb{Z}}(-,\mathbb{Q/Z})$ is also a contravariant functor from $\Modu(R^{\rm op})$ to $\Modu(R)$.  

A duality pair $(\mathcal{M,N})$ over $R$ is said to be:
\begin{itemize}
\item \emph{(co)product-closed} if $\mathcal{M}$ is closed under (co)products;

\item \emph{perfect} if it is coproduct closed, $\mathcal{M}$ is closed under extensions and contains ${}_R R$ (the ring $R$ regarded as an  $R$-module);

\item \cite[Appx. A]{BGH14} \emph{symmetric} if $(\mathcal{N,M})$ is a duality pair over $R^{\rm op}$;

\item \cite[Appx. A]{BGH14} \emph{complete} if it is perfect and symmetric;

\item \cite[Def. 3.2]{WangDi20} \emph{bicomplete} if it is complete and $({}^\perp\mathcal{N},\mathcal{N})$ is a hereditary cotorsion pair in $\Modu(R^{\rm op})$ cogenerated by a set;

\item \cite[Def. 3.5]{BecerrilPerez25} \emph{Tor-orthogonal} if $\Tor^R_{\geq 1}(\mathcal{N,M}) = 0$.
\end{itemize}
Some of the previous  types of duality pairs are important since they induce approximations by the classes $\mathcal{M}$ and $\mathcal{N}$. Indeed, it was proven in \cite[Thm.  3.1]{HJ09} that:
\begin{itemize}
\item $\mathcal{M}$ is closed under pure submodules, pure quotients and pure extensions.

\item If $(\mathcal{M,N})$ is coproduct closed, then $\mathcal{M}$ is covering.

\item If $(\mathcal{M,N})$ is product closed, then $\mathcal{M}$ is preenveloping.

\item If $(\mathcal{M,N})$ is perfect, then $(\mathcal{M},\mathcal{M}^{\perp_1})$ is a perfect cotorsion pair (that is, a cotorsion pair such that $\mathcal{M}$ is covering and $\mathcal{M}^{\perp_1}$ is enveloping). Such cotorsion pairs are, in particular, complete. 
\end{itemize}
In \cite[Prop. 2.3]{GillespieDuality} it is proven by Gillespie that if $(\mathcal{M,N})$ is a perfect duality pair, then $\mathcal{P}(R) \subseteq \mathcal{M}$ and $\mathcal{I}(R^{\rm op}) \subseteq \mathcal{N}$. Moreover, $\mathcal{M}$ is closed under direct limits and so $\mathcal{F}(R) \subseteq \mathcal{M}$ by Lazard's Theorem. 

From \cite{HJ09} we know that $([\Flat(R)]^\wedge_m,[\Inj(R^{\rm op})]^\vee_m)$ is a perfect duality pair over $R$ for every $m \in \mathbb{Z}_{\geq 0}$. If $R$ is right coherent (resp., right noetherian), then the previous pair is product closed (resp., symmetric). Another example of complete duality pair which generalizes the previous one is obtained from modules of finite type. These modules and their corresponding Ext and Tor-orthogonal complements were studied in \cite{BP17} by Bravo and the second named author. An $R^{\rm op}$-module $F$  is \emph{of type} $\text{FP}_n$ if there exists an exact sequence
\[
P_n \to P_{n-1} \to \cdots \to P_1 \to P_0 \to F \to 0
\]
where each $P_k$ is finitely generated and free. We denote by $\FP_n(R^{\rm op})$ the class formed by these modules. Note that if $n = 0$ (resp., $n = 1$) we obtain the class of finitely generated (resp., finitely presented) $R^{\rm op}$-modules. Moreover, the class of $R^{\rm op}$-modules of type $\text{FP}_\infty$ is defined as the intersection 
\[
\FP_\infty(R^{\rm op}) := \bigcap_{n \geq 0} \FP_n(R^{\rm op}).
\] 
The orthogonal complements $[\FP_n(R^{\rm op})]^{\perp_1}$ and $[\FP_n(R^{\rm op})]^{\top_1}$ are the classes of $\text{FP}_n$-injective $R^{\rm op}$-modules and of $\text{FP}_n$-flat $R$-modules, which we denote by $\Inj_n(R^{\rm op})$ and $\Flat_n(R)$, respectively. If $n > 1$, it is shown in \cite{BP17} that $(\Flat_n(R),\Inj_n(R^{\rm op}))$ is a complete duality pair. More examples of complete duality pairs can be found in \cite[Ex. 3.10]{BecerrilPerez25}.


\subsection*{Semidualizing bimodules}

Let $R$ and $S$ be associative rings with identities. In \cite{ArayaTakahashiYoshino05} Araya, Takahashi and Yoshino introduced the notion semidualizing bimodules. An $(S, R)$-bimodule $C = {}_S C_R$ is \emph{semidualizing} if:
\begin{enumerate}
\item[$\mathsf{(sd1)}$] ${}_S C \in \FP_\infty(S)$.
\item[$\mathsf{(sd2)}$] $C_R \in \FP_\infty(R^{\rm op})$.
\item[$\mathsf{(sd3)}$] The homothety map ${}_S S_S \to \Hom_{R^{\op}}(C,C)$ is an isomorphism.
\item[$\mathsf{(sd4)}$] The homothety map ${}_R R_R \to \Hom_S (C,C)$ is an isomorphism. 
\item[$\mathsf{(sd5)}$] $\Ext_S^{\geq 1}(C,C) = 0$.
\item[$\mathsf{(sd6)}$] $\Ext_{R^{\op}}^{\geq 1}(C,C) = 0$.
\end{enumerate}
Examples of semidualizing bimodules include free $R$-modules of rank one, ${}_\Lambda D(\Lambda)_\Lambda$ where $\Lambda$ is a finite dimensional algebra over a field $\mathbb{K}$ and $D(\Lambda) = \Hom_{\mathbb{K}}(\Lambda,\mathbb{K})$, and the \emph{generalized tilting modules} (a.k.a. \emph{Wakamatsu tilting}) introduced in \cite{Wakamatsu88}.

A semidualizing bimodule $_S C _R$ is called \emph{faithfully semidualizing} if it satisfies the following conditions for every $M \in \Modu(S)$ and $N \in \Modu(R^{\rm op})$:
\begin{enumerate}
\item[$\mathsf{(fsd7)}$] If $\Hom_S (C,M) = 0$, then $M = 0$,
\item[$\mathsf{(fsd8)}$] If $\Hom_{R^{\op}} (C,N) = 0$, then $N = 0$. 
\end{enumerate}

Associated to a semidualizing bimodule $_S C_R$ we have the Auslander and Bass classes, introduced by Foxby in \cite{Foxby73}. The \emph{Auslander class} $\A_C(R)$ with respect to ${}_S C_R$ consists of all $R$-modules $M \in \Modu (R)$ satisfying the following conditions:
\begin{enumerate}
\item[$\mathsf{(a1)}$] $\Tor^{R}_{\geq 1}(C,M ) = 0 $.
\item[$\mathsf{(a2)}$] $\Ext_S^{\geq 1}(C, C \otimes_R M) = 0$.
\item[$\mathsf{(a3)}$] The natural evaluation map $\mu_{M} \colon {}_R M \to \Hom_S(C, C \otimes_R M)$ is an isomorphism in $\Modu(R)$.
\end{enumerate} 
The \emph{Bass class} $\B_C(S)$ with respect to  ${}_S C_R$ consists of all modules $N \in \Modu (S)$ satisfying the following conditions:
\begin{enumerate}
\item[$\mathsf{(b1)}$] $\Ext_S^{\geq 1}(C, N) = 0$.
\item[$\mathsf{(b2)}$] $\Tor_{\geq 1}^R(C, \Hom_S(C,N)) = 0$.
\item[$\mathsf{(b3)}$] The natural evaluation map $\nu_N \colon C \otimes_R \Hom_S (C, N) \to N$ is an isomorphism in $\Modu (S)$.
\end{enumerate} 

Note that $C$ is a semidualizing $(S,R)$-bimodule if, and only if, it is a semidualizing $(R^{\rm op},S^{\rm op})$-bimodule. We can thus also consider the corresponding Auslander and Bass classes $\A_C(S^{\rm op})$ and $\B_C(R^{\rm op})$. 

An important and useful property of the Auslander and Bass classes proved in \cite[Prop. 4.1]{HolmWhite07} is the existence of the following equivalences of categories:
\[
\begin{tikzpicture}[description/.style={fill=white,inner sep=2pt}] 
\matrix (m) [ampersand replacement=\&, matrix of math nodes, row sep=3em, column sep=6em, text height=1.25ex, text depth=0.25ex] 
{ 
\A_C(R) \& \B_C(S) \& \text{$\mathsf{[first}$ $\mathsf{Foxby}$ $\mathsf{equivalence]}$} \\
}; 
\path[->] 
(m-1-1) edge [bend left = 20] node[above] {\footnotesize$C \otimes_R \sim$} node[below] {\footnotesize$\sim$} (m-1-2)
(m-1-2) edge [bend left = 20] node[above] {\footnotesize$\sim$} node[below] {\footnotesize$\Hom_S(C,\sim)$} (m-1-1)
;
\end{tikzpicture} 
\]
\[
\begin{tikzpicture}[description/.style={fill=white,inner sep=2pt}] 
\matrix (m) [ampersand replacement=\&, matrix of math nodes, row sep=3em, column sep=6em, text height=1.25ex, text depth=0.25ex] 
{ 
\A_C(S^{\rm op}) \& \B_C(R^{\rm op}) \& \text{$\mathsf{[second}$ $\mathsf{Foxby}$ $\mathsf{equivalence]}$} \\
}; 
\path[->] 
(m-1-1) edge [bend left = 20] node[above] {\footnotesize$- \otimes_S C$} node[below] {\footnotesize$\sim$} (m-1-2)
(m-1-2) edge [bend left = 20] node[above] {\footnotesize$\sim$} node[below] {\footnotesize$\Hom_{R^{\rm op}}(C,\sim)$} (m-1-1)
;
\end{tikzpicture} 
\]
It was recently proved by Huang in \cite[Thm. 3.3]{Huang21} that $(\A_C(R),\B_C(R^{\op}))$ is a perfect product-closed duality pair over $R$, and that $(\B_C(R^{\op}),\A_C(R))$ is a (co)product-closed duality pair over $R^{\rm op}$. In particular, $(\A_C(R),\B_C(R^{\op}))$ is a complete duality pair. Similarly, $(\A_C(S^{\op}),\B_C(S))$ is a complete duality pair over $S^{\rm op}$ with $\B _C (S)$ closed under (co)products. Moreover, we can note by \cite[Thm. 6.2]{HolmWhite07} that $\B_C(S)$ is always coresolving and $\A_C (S^{\op})$ is always resolving. If in addition ${}_S C_R$ is faithfully semidualizing, by \cite[Coroll. 6.3]{HolmWhite07} $\B_C(S)$ is closed under kernels of epimorphisms and $\A_C (S^{\op})$ under cokernels of monomorphisms, and so they are \emph{thick} (that is, closed under extensions, kernels of epimorphisms, cokernels of monomorphisms and direct summands). Similar comments apply to $\A_C(R)$ and $\B_C(R^{\op})$.


\subsection*{Abelian model structures}

Recall from \cite{Hovey02} that an abelian model structure on $\Modu (R)$ is formed by three classes of $R$-homomorphisms, called fibrations, cofibrations and weak equivalences, such that: an $R$-homomorphism $f$ is a cofibration (resp., fibration) if, and only if, it is a monomorphism (resp., an epimorphism) such that $\Coker(f)$ is a cofibrant $R$-module (resp., $\Ker(f)$ is fibrant). We recommend \cite{Hovey99} as a nice introduction to the subject of model categories. There is an appealing one-to-one correspondence between abelian model structures and cotorsion pairs, proved in \cite[Thm. 2.2]{Hovey02}. Later in \cite[Thm. 1.2]{Gillespie15}, Gillespie proved that if $(\mathcal{X},\mathcal{Y}')$ and $(\mathcal{X}',\mathcal{Y})$ are hereditary complete cotorsion pairs in $\Modu (R)$ (or in any weakly idempotent complete exact category) such that $\mathcal{X}' \subseteq \mathcal{X}$, $\mathcal{Y}' \subseteq \mathcal{Y}$ and $\mathcal{X} \cap \mathcal{Y}' = \mathcal{X}' \cap \mathcal{Y}$, then there exists a thick class $\mathcal{W}$ and a unique abelian model structure on $\Modu (R)$ such that $\mathcal{X}$, $\mathcal{Y}$ and $\mathcal{W}$ are the classes of cofibrant, fibrant and trivial $R$-modules, respectively. The model structures found in this paper are obtained by this method.


\section{Induced duality pairs} \label{sec:induced}

In \cite{WuGao22}, Wu and Gao find a relation between the duality pairs $(\Flat_n(R),\Inj_n(R^{\rm op}))$ and $(\A_C(R),\B_C(R^{\rm op}))$. More specifically, given a semidualizing bimodule ${}_S C_R$, they first show in \cite[Prop. 3.3]{WuGao22} that $\Flat_n(R) \subseteq \A_C(R)$ and $\Inj_n(S) \subseteq \B_C(S)$. This allows to restrict the Foxby equivalence $C \otimes_R \sim$ and its quasi-inverse $\Hom_S(C,\sim)$ to the subcategories $\Flat_n(R)$ and $\Inj_n(S)$, respectively. Following the notation in \cite{WuGao22}, let $\mathcal{FF}^n_C(S)$ and $\mathcal{FI}^n_C(R)$ denote the essential images of $\Flat_n(R)$ and $\Inj_n(S)$ under $C \otimes_R \sim$ and $\Hom_S(C,\sim)$, that is, 
\begin{align*}
\mathcal{FF}^n_C(S) & := \{ M \in \Modu(S) \ {\rm : } \ M \simeq C \otimes_R F \text{ where } F \in \Flat_n(R) \}, \\
\mathcal{FI}^n_C(R) & := \{ N \in \Modu(R) \ {\rm : } \ N \simeq \Hom_S(C,E) \text{ where } E \in \Inj_n(S) \},
\end{align*}
(where the notation $\simeq$ means that two objects are isomorphic). It turns out from \cite[Prop. 4.1]{WuGao22} that there is a equivalence of categories between $\Flat_n(R)$ and $\mathcal{FF}^n_C(S)$, and between $\mathcal{FI}^n_C(R)$ and $\Inj_n(S)$, given by the restrictions $(C \otimes_R \sim)|_{\Flat_n(R)}$ and $\Hom_S(C,\sim) |_{\Inj_n(S)}$, respectively. Moreover, in \cite[\S 3]{WuGao22} the authors show that 
\[
(\mathcal{FF}^n_C(S),\mathcal{FI}^n_C(S^{\rm op})) \text{ \ \ and \ \ } (\mathcal{FI}^n_C(S^{\rm op}),\mathcal{FF}^n_C(S))
\] 
are (co)product closed duality pairs over $S$ and $S^{\rm op}$, respectively. In addition, if ${}_S S \in \mathcal{FF}^n_C(S)$ (resp., $S_S \in \mathcal{FI}^n_C(S^{\rm op})$) then 
\[
(\mathcal{FF}^n_C(S),\mathcal{FI}^n_C(S^{\rm op})) \ \ (\text{resp., } (\mathcal{FI}^n_C(S^{\rm op}),\mathcal{FF}^n_C(S)))
\] 
is a complete duality pair. In other words, these pairs inherit the properties of the duality pairs $(\Flat_n(R),\Inj_n(R^{\rm op}))$ and $(\Flat_n(S^{\rm op}),\Inj_n(S))$. 

The main goal of this section is to generalize the previous result to an arbitrary duality pair $(\mathcal{M,N})$ over $R$. The idea is to explore the possibility of inducing from $(\mathcal{M,N})$ via Foxby equivalences a duality pair whose halves are contained in the Auslander and Bass classes $\A_C(R)$ and $\B_C(R^{\rm op})$. One problem with this is that not every duality pair $(\mathcal{M,N})$ over $R$ satisfies the containments $\mathcal{M} \subseteq \A_C(R)$ and $\mathcal{N} \subseteq \B_C(R^{\rm op})$. However, we can overcome this constraint by considering the intersections $\mathcal{M} \cap \A_C(R)$ and $\mathcal{N} \cap \B_C(R^{\rm op})$. These intersections are not trivial if $(\mathcal{M,N})$ is perfect (as many duality pairs of interest), since in this situation one has $\Flat(R) \subseteq \mathcal{M} \cap \A_C(R)$ and $\Inj(R^{\rm op}) \subseteq \B_C(R^{\rm op})$. 
 
From now on, $C = {}_S C_R$ is a semidualizing $(S,R)$-bimodule and $(\mathcal{M,N})$ is a duality pair over $R$. We consider the following induced classes of $S$-modules and $S^{\rm op}$-modules:
\begin{align*}
\mathcal{M}^C(S) & := \{ M \in \Modu(S) \ {\rm : } \ M \simeq C \otimes_R A \text{ where } A \in \A_C(R) \cap \mathcal{M} \}, \\
\mathcal{N}^C(S^{\rm op}) & := \{ N \in \Modu(S^{\rm op}) \ {\rm : } \ N \simeq \Hom_{R^{\rm op}}(C,B) \text{ where } B \in \B_C(R^{\rm op}) \cap \mathcal{N} \}.
\end{align*}

\begin{remark}
Note that when $(\mathcal{M,N}) = (\Flat(R), \Inj(R^{\op}))$, $\mathcal{M}^C(S)$ and $\mathcal{N}^C(S^{\rm op})$ are precisely the classes $\Flat_C(S)$ and $\Inj_C( S^{\op})$ of \cite[Def. 5.1]{HolmWhite07}. In \cite[\S 5]{HolmWhite07} the authors show which properties are inherited from $(\Flat(R), \Inj(R^{\op}))$ by these classes $\Flat_C(S)$ and $\Inj_C( S^{\op})$. Moreover, we also obtain the classes $\mathcal{FF}^n_C(S)$ and $\mathcal{FI}^n_C(S^{\rm op})$ after setting $(\mathcal{M,N}) = (\Flat_n(R), \Inj_n(R^{\op}))$.
\end{remark}

Following the characterizations of the classes $\Flat_C(S)$ and $\Inj_C( S^{\op})$ shown in \cite[Lem. 5.1]{HolmWhite07}, and of $\mathcal{FF}^n_C(S)$ and $\mathcal{FI}^n_C(S^{\rm op})$ obtained in \cite[Lem. 3.5]{WuGao22}, we have the following descriptions of the induced classes $\M^{C}(S)$ and $\mathcal{N}^{C}(S^{\op})$ (the proof is straightforward).

\begin{lemma}[characterizations of the induced classes] \label{Lema1}
The following assertions hold for every $M \in \Modu(S)$ and $N \in \Modu(S^{\rm op})$:
\begin{enumerate}
\item $M \in \M^{C}(S)$ if, and only if, $M \in \B_C(S)$ and $\Hom_S(C, M) \in \M$.

\item $N \in \mathcal{N}^{C}(S^{\op})$ if, and only if, $N \in \A_C(S^{\op})$ and $N \otimes_{S} C \in \N$.
\end{enumerate}
\end{lemma}

The following result is a relative version of \cite[Prop. 4.1]{WuGao22} (equivalences between $\Flat_n(R)$ and $\mathcal{FF}^n_C(S)$, and between $\mathcal{FI}^n_C(R)$ and $\Inj_n(S)$). Its proof is similar to that provided by Wu and Gao, and its a consequence of the previous lemma.

\begin{proposition}[restricted Foxby equivalences] \label{Equivalencia1}
There are equivalences of categories 
\[
\begin{tikzpicture}[description/.style={fill=white,inner sep=2pt}] 
\matrix (m) [ampersand replacement=\&, matrix of math nodes, row sep=3em, column sep=4.5em, text height=1.25ex, text depth=0.25ex] 
{ 
\M \cap \A_C(R) \& \M^C(S) \& \N^C(S^{\rm op}) \& \N \cap \B_C(R^{\rm op}) \\
}; 
\path[->] 
(m-1-1) edge [bend left = 20] node[above] {\footnotesize$C \otimes_R \sim$} node[below] {\footnotesize$\sim$} (m-1-2)
(m-1-2) edge [bend left = 20] node[above] {\footnotesize$\sim$} node[below] {\footnotesize$\Hom_S(C,\sim)$} (m-1-1)
(m-1-3) edge [bend left = 20] node[above] {\footnotesize$- \otimes_{S} C$} node[below] {\footnotesize$\sim$} (m-1-4)
(m-1-4) edge [bend left = 20] node[above] {\footnotesize$\sim$} node[below] {\footnotesize$\Hom_{R^{\rm op}}(C,\sim)$} (m-1-3)
;
\end{tikzpicture} 
\]
\end{proposition}

The following result generalizes \cite[Prop. 3.13 (1)]{WuGao22}. We provide a proof for the reader's convenience and because of the importance of this result for the rest of the paper.

\begin{proposition}[induced duality pair]\label{Dualidad1}
$(\M^C(S),\N^C(S^{\rm op}))$ is a duality pair over $S$. In particular, $\M^C(S)$ is closed under pure submodules, pure quotients and pure extensions. Moreover, if $(\M,\N)$ is symmetric, then so is $(\M^C(S),\N^C(S^{\rm op}))$.
\end{proposition}

\begin{proof}
Let us first show that $M \in \M^C(S)$ if, and only if, $M^+ \in \N^C(S^{\rm op})$. Suppose that $M \in \M^C(S)$. Then, $M \simeq C \otimes_R A$ for some $A \in \A_C(R) \cap \M$. By \cite[Thm. 3.2.1]{EnochsJenda00}, we have the natural isomorphism
\[
M^+ \simeq (C \otimes_R A)^+ \simeq \Hom_{R^{\rm op}}(C,A^+),
\]
where $A^+ \in \B_C(R^{\rm op}) \cap \N$ since $(\A_C(R),\B_C(R^{\rm op}))$ and $(\mathcal{M,N})$ are duality pairs over $R$. It then follows that $M^+ \in \N^C(S^{\rm op})$. Now assume that $M^+ \in \N^C(S^{\rm op})$, that is, $M^+ \simeq \Hom_{R^{\rm op}}(C,B)$ for some $B \in \B_C(R^{\rm op}) \cap \N$. Using the second Foxby equivalence, we have that $M^+ \in \A_C(S^{\rm op})$ and that 
\[
M^+ \otimes_S C \simeq \Hom_{R^{\rm op}}(C,B) \otimes_S C \simeq B.
\]
Now, since  $C$ is finitely presented $S$-module by $\mathsf{(sd1)}$, the previous and \cite[Thm. 3.2.11]{EnochsJenda00} imply that
\[
[\Hom_S(C,M)]^+ \simeq M^+ \otimes_S C \simeq B \in \B_C(R^{\op}) \cap \N.
\]
On the other hand, since $M^+ \in \A_C (S^{\op})$, by the second Foxby equivalence, \cite[Thm. 3.2.1]{EnochsJenda00} again, and the previous natural isomorphism we obtain 
\[
M^+ \simeq \Hom _{R^{\op}}(C, M^+ \otimes_S C) \simeq \Hom _{R^{\op}}(C, [\Hom_S(C,M)]^+) \simeq [C \otimes_R \Hom_S(C, M)]^+. 
\]
Since $\mathbb{Q/Z}$ is a cogenerator in the category $\Modu(\mathbb{Z})$ of abelian groups, the isomorphism $M^+ \simeq [C \otimes_R \Hom_S(C, M)]^+$ yields $M \simeq C \otimes_R \Hom_S(C, M)$, where $\Hom_S(C, M) \in \A_C(R) \cap \M$ since $[\Hom_S(C,M)]^+ \in \B_C(R^{\op}) \cap \N$ and $(\mathcal{M,N})$ and $(\A_C(R),\B_C(R^{\rm op}))$ are duality pairs over $R$. Hence, $M \in \M^C(S)$.  

Using similar arguments, one can show that if $(\M,\N)$ is symmetric, then so is $(\M^C(S),\N^C(S^{\rm op}))$. For one uses the fact that $(\A_C(R),\B_C(R^{\rm op}))$ is a complete (and so symmetric) duality pair. 

It remains to prove that $\N^{C}(S^{\op})$ is closed under direct summands and finite direct sums, but this follows from standard homological algebra arguments.
\end{proof}

\begin{remark}
Note that the previous result provides a more direct proof \cite[Prop. 3.11]{WuGao22}. 
\end{remark}

If we assume additional properties for the duality pair $(\mathcal{M,N})$, in the following results we explore which of these properties can be inherited by the duality pair $(\M^{C}(S), \N^{C} (S^{\op}))$.

\begin{proposition}[inherited properties] \label{Prop0}
The followings statements hold:
\begin{enumerate}
\item If $\M$ (resp., $\N$) is closed under extensions, then so is $\M^{C} (S)$ (resp., $\N^{C}(S^{\op})$).
\item If $\M$ (resp., $\N$) is closed under direct products and direct limits (in particular, coproducts) then so is $\M^C(S)$ (resp., $\N^C(S^{\op})$).
\item If $(\M,\N)$ is (co)product closed, then so is $(\M^C(S),\N^C(S^{\rm op}))$, and $\M^C(S)$ is preenveloping (resp., covering).
\item If $(\M,\N)$ is perfect and $S \in \M^C(S)$, then so is $(\M^C(S),\N^C(S^{\rm op}))$, and $(\M^C(S),[\M^C(S)]^{\perp_1})$ is a perfect cotorsion pair. 
\item Consider the following assertions: 
\begin{enumerate}[(a)]
\item $\M^C(S)$ is closed under kernels of epimorphisms.
\item $\M \cap \A _C(R)$ is closed under kernels of epimorphisms.
\item $\N^C(S^{\rm op})$ is closed under cokernels of monomorphisms.
\item $\N \cap \B_C(R^{\rm op})$ is closed under cokernels of monomorphisms.
\end{enumerate}
If ${}_SC_R$ is faithfully semidualizing, then the implications (a) $\Leftrightarrow$ (b) and (c) $\Leftrightarrow$ (d) hold. If in addition $(\M,\N)$ is symmetric, then all the assertions are equivalent. 
\item If $(\M,\N)$ is complete and $S \in \M^C(S)$, then $(\M^C(S),\N^C(S^{\rm op}))$ is complete. 
\end{enumerate}
\end{proposition}

\begin{proof} \mbox{}
\begin{enumerate}
\item Follows as in \cite[Prop. 3.6]{WuGao22}.

\item These closure properties follow by applying Lemma \ref{Lema1}, the facts that $\B_C (S)$ and  $\A_C(S^{\rm op})$ are closed under direct summands, direct products and direct limits (see \cite[Prop. 4.2]{HolmWhite07}), that the functors $\Hom_S(C,\sim)$ and $- \otimes_{S} C$ preserves direct limits and direct products, since $C \in \FP_\infty(S)$ (see Brown's \cite[Coroll. after Thm. 1]{Brown75}).

\item Follows from (2) and the properties of duality pairs (see \S \ref{sec:prelims}).

\item Follows from (1), (3) and the properties of duality pairs. 

\item The equivalences (a) $\Leftrightarrow$ (b) and (c) $\Leftrightarrow$ (d) follow as in \cite[Thms. 3.7 and 3.8 (1) $\Leftrightarrow$ (2)]{WuGao22}. Now suppose that $(\M,\N)$ is symmetric. We show (b) $\Leftrightarrow$ (d). We know that $(\A _C(R),\B_C(R^{\rm op}))$ is a complete duality pair, then the intersection pair $(\M \cap \A _C(R), \N \cap \B_C(R^{\rm op}))$ is a symmetric duality pair by \cite[Prop. 2.8]{HJ09}. So one can note that $\M \cap \A _C(R)$ is closed under kernels of epimorphisms if, and only if, $\N \cap \B_C(R^{\rm op})$ is closed under cokernels of monomorphisms, using the properties of $\Hom_{\mathbb{Z}}(-,\mathbb{Q/Z})$. 

\item Follows from (4) and Proposition \ref{Dualidad1}. 
\end{enumerate}
\end{proof}

\begin{remark}
In the previous proposition, (1) and (2) generalize \cite[Props. 3.6 and 3.12]{WuGao22}, (3) and (4) generalize \cite[Thm. 3.15 (1) and Coroll. 3.16 (1)]{WuGao22}, and (5) generalizes \cite[Thms. 3.7 \& 3.8]{WuGao22}.
\end{remark}

Another consequence of assuming that $_S C _R$ is faithfully semidualizing is to extend the equivalences appearing in Proposition \ref{Equivalencia1}. This is specified and proved in Proposition \ref{extended-equivalence} below. Before that, note that since $\B_C(S)$ is coresolving and $\A_C (S^{\op})$ is resolving, then the containments 
\[
[\M^C(S)]^\wedge_n \subseteq \B_C(S) \ \ \text{and} \ \ [\N^C(S^{\op})]^\vee_n \subseteq \A_C (S^{\op})
\] 
hold true for every $n \in \mathbb{Z}_{\geq 0}$, and are consequences of Lemma \ref{Lema1}. On the other hand, since ${}_S C_R$ is faithfully semidualizing, $\B_C(S)$ and $\A_C (S^{\op})$ are thick. Thus, we also get the containments 
\[
[\M \cap \A_C(R)]^\wedge_n \subseteq \A_C(R) \ \ \text{and} \ \ [\N \cap \B_C(R^{\rm op})]^\vee_n \subseteq \B_C(R^{\rm op})
\] 
for every $n \in \mathbb{Z}_{\geq 0}$.

\begin{proposition}[relative homological dimensions] \label{Resdim}
The inequalities 
\[
\resdim_{M \cap \A_C(R)}(M) \leq \resdim_{\M^{C}(S)}(C \otimes _R M)
\]
and
\[
\coresdim_{\N \cap \B_C(R^{\op})}(N) \leq \coresdim_{\N^C(S^{\op})}(\Hom_S(C,N))
\]
hold for every $M \in \Modu (R)$ and $N \in \Modu (R^{\op})$. If in addition $_S C _R$ is faithfully semidualizing, then the previous inequalities become equalities. 
\end{proposition}

\begin{proof}
First, note by Lemma \ref{Lema1} and the fact that $\B_C (S)$ is coresolving, that any $\M^C(S)$-resolution of $C \otimes_R M$ of finite length is an exact complex with cycles in $\B_C(S)$. The previous, along with $\mathsf{(b1)}$ and Proposition \ref{Equivalencia1}, imply that the equivalence $\Hom_S(C,\sim)$ maps $\M^C(S)$-resolutions of $C \otimes_R M$ of length $n$ to $(\M \cap \A_C(R))$-resolutions of $M$ of the same length. Then, we have that 
\[
\resdim_{M \cap \A_C(R)}(M) \leq \resdim_{\M^{C}(S)}(C \otimes _R M).
\]

Suppose now that $_S C _R$ is faithfully semidualizing. Then any $(\M \cap \A_C(R))$-resolution of $M$ of finite length is an exact complex with cycles in $\A_C(R)$, since $\A_C(R)$ becomes thick. It follows that the equivalence $C \otimes_R \sim$ maps $(\M \cap \A_C(R))$-resolutions of $M$ of length $n$ to $\M^C(S)$-resolutions of $C \otimes_R M$ of length $n$. In other words, $C \otimes_R \sim$ is an exact functor from $[\M \cap \A_C(R)]^\wedge_n$ to $[\M^C(S)]^\wedge_n$. Then, we also have that 
\[
\resdim_{\M^{C}(S)}(C \otimes _R M) \leq \resdim_{M \cap \A_C(R)}(M).
\]

The inequality and equality for coresolution dimensions follow similarly. 
\end{proof}

\begin{proposition}[extended Foxby equivalences] \label{extended-equivalence}
If $_S C _R$ is faithfully semidualizing, then for every $n \in \mathbb{Z}_{\geq 0}$ there are equivalences of categories
\[
\begin{tikzpicture}[description/.style={fill=white,inner sep=2pt}] 
\matrix (m) [ampersand replacement=\&, matrix of math nodes, row sep=6em, column sep=3.5em, text height=1.25ex, text depth=0.25ex] 
{ 
{[\M \cap \A_C(R)]}^\wedge_n \& {[\M^C(S)]}^\wedge_n \& {[\N^C(S^{\rm op})]}^\vee_n \& {[\N \cap \B_C(R^{\rm op})]}^\vee_n \\
}; 
\path[->] 
(m-1-1) edge [bend left = 25] node[above] {\footnotesize$C \otimes_R \sim$} node[below] {\footnotesize$\sim$} (m-1-2)
(m-1-2) edge [bend left = 25] node[above] {\footnotesize$\sim$} node[below] {\footnotesize$\Hom_S(C,\sim)$} (m-1-1)
(m-1-3) edge [bend left = 25] node[above] {\footnotesize$- \otimes_{S} C$} node[below] {\footnotesize$\sim$} (m-1-4)
(m-1-4) edge [bend left = 25] node[above] {\footnotesize$\sim$} node[below] {\footnotesize$\Hom_{R^{\rm op}}(C,\sim)$} (m-1-3)
;
\end{tikzpicture} 
\]
\end{proposition}

\begin{proof}
Since $\resdim_{M \cap \A_C(R)}(M) = \resdim_{\M^{C}(S)}(C \otimes _R M)$ by Proposition \ref{Resdim}, we have a functor $C \otimes_R \sim \colon {[\M \cap \A_C(R)]}^\wedge_n \to {[\M^C(S)]}^\wedge_n$. Now let us show that the functor $\Hom_S(C,\sim)$ maps ${[\M^C(S)]}^\wedge_n$ into ${[\M \cap \A_C(R)]}^\wedge_n$. Let $M \in {[\M^C(S)]}^\wedge_n \subseteq \B_C(S)$. Since $C \otimes_R \sim \colon \A_C(R) \to \B_C(S)$ is essentially surjective, there exists $A \in A_C(R)$ such that $M \simeq C \otimes_R A$. Then, $A \simeq \Hom_S(C,C\otimes_R A) \simeq \Hom_S(C,M)$ and by Proposition \ref{Resdim} we obtain
\[
\resdim_{[\M \cap \A_C(R)]}(\Hom_S(C,M)) = \resdim_{\M^C(S)}(M) \leq n.
\]
Hence, $C \otimes_R \sim$ defines an equivalence from ${[\M \cap \A_C(R)]}^\wedge_n$ to ${[\M^C(S)]}^\wedge_n$ with quasi-inverse $\Hom_S(C,\sim)$. The equivalences in the other diagram can be established by using an analogous procedure. 
\end{proof}

\begin{corollary}
If $_S C _R$ is faithfully semidualizing, then in the following diagram, the horizontal arrows are equivalences and the vertical arrows are inclusions of categories: 
\[
\begin{tikzpicture}[description/.style={fill=white,inner sep=2pt}] 
\matrix (m) [ampersand replacement=\&, matrix of math nodes, row sep=6em, column sep=3.5em, text height=1.25ex, text depth=0.25ex] 
{ 
\M \cap \A_C(R) \& \M^C(S) \& \N^C(S^{\rm op}) \& \N \cap \B_C(R^{\rm op}) \\
{[\M \cap \A_C(R)]}^\wedge_n \& {[\M^C(S)]}^\wedge_n \& {[\N^C(S^{\rm op})]}^\vee_n \& {[\N \cap \B_C(R^{\rm op})]}^\vee_n \\
\A_C(R) \& \B_C(S) \& \A_C(S^{\rm op}) \& \B_C(R^{\rm op}) \\
}; 
\path[->] 
(m-1-1) edge [bend left = 25] node[above] {\footnotesize$C \otimes_R \sim$} node[below] {\footnotesize$\sim$} (m-1-2)
(m-1-2) edge [bend left = 25] node[above] {\footnotesize$\sim$} node[below] {\footnotesize$\Hom_S(C,\sim)$} (m-1-1)
(m-1-3) edge [bend left = 25] node[above] {\footnotesize$- \otimes_{S} C$} node[below] {\footnotesize$\sim$} (m-1-4)
(m-1-4) edge [bend left = 25] node[above] {\footnotesize$\sim$} node[below] {\footnotesize$\Hom_{R^{\rm op}}(C,\sim)$} (m-1-3)
(m-2-1) edge [bend left = 25] node[above] {\footnotesize$C \otimes_R \sim$} node[below] {\footnotesize$\sim$} (m-2-2)
(m-2-2) edge [bend left = 25] node[above] {\footnotesize$\sim$} node[below] {\footnotesize$\Hom_S(C,\sim)$} (m-2-1)
(m-2-3) edge [bend left = 25] node[above] {\footnotesize$- \otimes_{S} C$} node[below] {\footnotesize$\sim$} (m-2-4)
(m-2-4) edge [bend left = 25] node[above] {\footnotesize$\sim$} node[below] {\footnotesize$\Hom_{R^{\rm op}}(C,\sim)$} (m-2-3)
(m-3-1) edge [bend left = 25] node[above] {\footnotesize$C \otimes_R \sim$} node[below] {\footnotesize$\sim$} (m-3-2)
(m-3-2) edge [bend left = 25] node[above] {\footnotesize$\sim$} node[below] {\footnotesize$\Hom_S(C,\sim)$} (m-3-1)
(m-3-3) edge [bend left = 25] node[above] {\footnotesize$- \otimes_{S} C$} node[below] {\footnotesize$\sim$} (m-3-4)
(m-3-4) edge [bend left = 25] node[above] {\footnotesize$\sim$} node[below] {\footnotesize$\Hom_{R^{\rm op}}(C,\sim)$} (m-3-3)
;
\path[>->]
(m-1-1) edge (m-2-1) (m-2-1) edge (m-3-1)
(m-1-2) edge (m-2-2) (m-2-2) edge (m-3-2)
(m-1-3) edge (m-2-3) (m-2-3) edge (m-3-3)
(m-1-4) edge (m-2-4) (m-2-4) edge (m-3-4)
;
\end{tikzpicture} 
\]
\end{corollary}

We close this section studying whether the Tor-orthogonality condition of a duality pair $(\M,\N)$ over $R$ can be transferred to the induced duality pair $(\M^C(S),\N^C(S^{\rm op}))$. Examples and non-examples of such duality pairs are provided in \cite{BecerrilPerez25}. The following is a corollary of \cite[Thm. 6.4]{HolmWhite07} that turns out to be a useful tool for some computations in \S \ref{sec:relGFGI}.

\begin{lemma}\label{lem:HolmWhite_for_MN}
For every $A, A' \in \mathcal{A}_C(R) \cap \mathcal{M}$ and $B, B' \in \mathcal{B}_C(R^{\rm op}) \cap \mathcal{N}$, there are the following natural isomorphisms for every $i \in \mathbb{Z}_{\geq 0}$:
\begin{enumerate}
\item $\Ext^i_S(C \otimes_R A', C \otimes_R A) \simeq \Ext^i_R(A',A)$.

\item $\Ext^i_{S^{\rm op}}(\Hom_{R^{\rm op}}(C,B),\Hom_{R^{\rm op}}(C,B')) \simeq \Ext^i_{R^{\rm op}}(B,B')$. 

\item $\Tor^S_i(\Hom_{R^{\rm op}}(C,B),C \otimes_R A) \simeq \Tor^R_i(B,A)$.
\end{enumerate}
\end{lemma}

The following is an immediate consequence of the previous lemma.

\begin{proposition}[inherited Tor-orthogonality]\label{prop:Tor-orthogonal}
The duality pair $(\M^C(S),\N^C(S^{\rm op}))$ is Tor-orthogonal if, and only if, the intersection pair $(\M \cap \A_C(R),\N \cap \B_C(R^{\rm op}))$ is Tor-orthogonal. In particular, if $(\M,\N)$ is Tor-orthogonal, then so is $(\M^C(S),\N^C(S^{\rm op}))$. 
\end{proposition}

\begin{example}
In general, the duality pair $(\mathcal{FF}^n_C(S),\mathcal{FI}^n_C(S^{\rm op}))$ is not Tor-orthogonal for $n > 1$. This follows by the previous proposition, \cite[Prop. 3.3]{WuGao22} and the fact that $(\mathcal{F}_n(R),\mathcal{I}_n(R^{\rm op}))$ is not Tor-orthogonal if $R$ is left $n$-coherent (by \cite[Thms. 5.5 \& 5.6]{BP17} and \cite[Ex. 3.6]{BecerrilPerez25}).
\end{example}


\section{Gorenstein flat and Gorenstein injective modules \\ from induced duality pairs}\label{sec:relGFGI}

In this section, we study homological and homotopical aspects, as well as duality relations via Foxby equivalences and $(-)^+$, of some versions of relative Gorenstein flat and relative Gorenstein injective modules induced by the duality pairs $(\M,\N)$ and $(\M^C(S),\N^C(S^{\rm op}))$. One part is motivated by the work of Cheng and Zhao \cite{ChengZhao24}. In their article, they propose the concept of Gorenstein $G_C$-$\text{FP}_n$-injective modules. These are defined as those $M \in \Modu (R)$ for which there exists an exact complex 
\[
Y_\bullet = \cdots \xrightarrow{\partial_2} W_1 \xrightarrow{\partial_1} W_0 \xrightarrow{\partial_0} I^0 \xrightarrow{\partial^0} I^1 \xrightarrow{\partial^1} \cdots
\]
such that $\Hom_R(H,Y_\bullet)$ is exact for every 
\[
H \in \mathcal{H}_C(\mathcal{I}_n(R)) = \{ \Hom_S(C,E) \ \text{:} \ E \in {}^{\perp_1}[\mathcal{I}_n(S)] \cap \mathcal{I}_n(S) \},
\]
and such that $W_k \in \mathcal{FI}^n_C(R)$ and $I^k \in \mathcal{I}_n(R)$ for every $k \in \mathbb{Z}_{\geq 0}$, with $M \simeq \Ker(\partial^0)$. We point out that these relative Gorenstein injective modules are defined from two classes coming from the complete duality pairs $(\mathcal{FF}^n_C(R^{\rm op}),\mathcal{FI}^n_C(R))$ and $(\mathcal{F}_n(S^{\rm op}),\mathcal{I}_n(S))$. On the other hand, some notions of Gorenstein flat and Gorenstien injective modules relative to duality pairs have been proposed recently, like for instance in \cite{WangYangZhu19} by Wang, Yang and Zhu, and in \cite{BecerrilPerez25}.  

One approach to relative Gorenstein injective objects is given in \cite[Def. 3.7]{BMS}. Given $\mathcal{Y}$ and $\mathcal{Z}$ classes of $R^{\rm op}$-modules, we say that $M \in \Modu (R^{\rm op})$ is $(\mathcal{Z,Y})$\emph{-Gorenstein injective} if $M \simeq \Ker(\partial^0)$ in an exact complex
\[
Y_\bullet = \cdots \to Y_1 \xrightarrow{\partial_1} Y_0 \xrightarrow{\partial_0} Y^0 \xrightarrow{\partial^0} Y^1 \xrightarrow{\partial^1} \cdots
\]
of $R^{\rm op}$-modules in $\mathcal{Y}$ such that $\Hom_{R^{\rm op}}(Z,Y_\bullet)$ is exact for every $Z \in \mathcal{Z}$. The class of $(\Z,\Y)$-Gorenstein injective $R^{\rm op}$-modules will be denoted by $\GI_{(\Z,\Y)}$. If $\Z = \Y$, we use the notation $\GI_{\Y}$. In particular, we can consider the classes $\GI_{\N^C(S^{\rm op})}$ and $\GI_{\N \cap \B_C(R^{\op})}$, and prove the following relation between them via Foxby equivalences.

From now on, we set 
\[
\mathfrak{M} := \mathcal{A}_C(R) \cap \mathcal{M} \ \ \text{and} \ \ \mathfrak{N} := \mathcal{B}_C(R^{\rm op}) \cap \mathcal{N}
\]
in order to simplify some of the upcoming notations.

\begin{proposition}[relative Gorenstein injectivity under the 2${}^{\rm nd}$ Foxby equivalence]\label{duality_Ginj}
If ${}_S C_R$ is faithfully semidualizing, then the following assertions hold for every $M \in \B_{C}(R^{\rm op})$ and $N \in \A_C(S^{\rm op})$:
\begin{enumerate}
\item $N \in \GI_{\N^C(S^{\rm op})}$ if, and only if $N \otimes_S C \in \GI_{\mathfrak{N}}$.

\item $M \in \GI_{\mathfrak{N}}$ if, and only if, $\Hom_{R^{\op}}(C,M) \in \GI_{\N^C(S^{\rm op})}$.
\end{enumerate} 
\end{proposition}

\begin{proof} 
Let us first show (1). For the ``only if'' part, let $N \in \A_C(S^{\rm op}) \cap \GI_{\N^C(S^{\rm op})}$. Then, $N \simeq \Ker(\partial^0)$ in an exact complex $Y_\bullet$ with components in $\N^C(S^{\rm op})$ such that $\Hom_{S^{\rm op}}(Y, Y_\bullet)$ is exact for all $Y \in \N^C(S^{\op})$. By Lemma \ref{Lema1} (2), $Y_\bullet$ is an exact complex with components in $\A_C(S^{\rm op})$ and $\Ker(\partial^0) \in \A_C(S^{\rm op})$. On the other hand, since ${}_{R^{\rm op}}C_{S^{\rm op}}$ is faithfully semidualizing, we have that $\A_C(S^{\rm op})$ is thick, and so $Y_\bullet$ is an exact complex with cycles in $\A_C(S^{\rm op})$. It then follows by $\mathsf{(a1)}$ and Proposition \ref{Equivalencia1} that $Y_\bullet \otimes_S C$ is an exact complex with components in $\mathfrak{N}$. It particular, $N \otimes_S C \simeq \Ker(\partial^0 \otimes_S C)$. It remains to show that $\Hom_{R^{\rm op}}(B,Y_\bullet \otimes_S C)$ is exact for every $B \in \mathfrak{N}$. By Lemma \ref{lem:HolmWhite_for_MN} (2) and the 2${}^{\rm nd}$ Foxby equivalence, we have the natural isomorphism $\Hom_{R^{\rm op}}(B,Y_\bullet \otimes_S C) \simeq \Hom_{S^{\rm op}}(\Hom_{R^{\rm op}}(C,B),Y_\bullet)$, where $\Hom_{S^{\rm op}}(\Hom_{R^{\rm op}}(C,B),Y_\bullet)$ is exact since $\Hom_{R^{\rm op}}(C,B) \in \N^C(S^{\rm op})$.

For the ``if'' part, suppose $N \otimes_S C \in \GI_{\mathfrak{N}}$. Then, $N \otimes_S C \simeq \Ker(\partial^0)$ in an exact complex $B_\bullet$ with components in $\mathfrak{N}$ such that $\Hom_{R^{\op}}(B,B_\bullet)$ is exact for all $B \in \mathfrak{N}$. On the other hand, since $\B_C(R^{\rm op})$ is thick, and $N \otimes_S C \in \B_C(R^{\rm op})$, we have that $B_\bullet$ is an exact complex with cycles in $\B_C(R^{\rm op})$, and so $\Hom_{R ^{\op}}(C,B_\bullet)$ is exact by $\mathsf{(b1)}$. Also, we get $N \simeq \Hom_{R ^{\op}} (C, N \otimes_S C) \simeq \Ker(\Hom_{R^{\rm op}}(C,\partial^0))$ since $N \in \A _C (S^{\op})$. It remains to show that $\Hom_{S^{\rm op}}(\Hom_{R^{\rm op}}(C,B),\Hom_{R ^{\op}}(C,B_\bullet))$ is exact for every $B \in \mathfrak{N}$, but this follows as in the ``only if'' part. 

For part (2), simply observe that $M \simeq \Hom_{R^{\rm op}}(C,M) \otimes_S C$, and so by part (1) we have that $M \in \GI_{\mathfrak{N}}$, if and only if, $\Hom_{R^{\op}}(C, M) \in \GI_{\N^C(S^{\rm op})}$.
\end{proof}

The following is immediate from the previous proposition.

\begin{corollary}[Gorenstein injective Foxby equivalence] \label{Gorenstein-injective-equivalence}
If $_S C _R$ is faithfully semidualizing, then there is an equivalence of categories
\[
\begin{tikzpicture}[description/.style={fill=white,inner sep=2pt}] 
\matrix (m) [ampersand replacement=\&, matrix of math nodes, row sep=6em, column sep=3.5em, text height=1.25ex, text depth=0.25ex] 
{ 
\A_C(S^{\op}) \cap \GI_{\N^C(S^{\rm op})} \&  \GI_{\mathfrak{N}} \cap \B_C(R^{\op}) \\
}; 
\path[->] 
(m-1-1) edge [bend left = 25] node[above] {\footnotesize$- \otimes_S C$} node[below] {\footnotesize$\sim$} (m-1-2)
(m-1-2) edge [bend left = 25] node[above] {\footnotesize$\sim$} node[below] {\footnotesize$\Hom_{R^{\rm op}}(C,\sim)$} (m-1-1)
;
\end{tikzpicture} 
\]
\end{corollary}

The previous proposition and corollary are our relative versions of \cite[Lems. 4.2, 4.4, Props. 4.3, 4.5, Thm. 4.6]{ChengZhao24}.

Regarding relative Gorenstein flat modules, if we are given $\mathcal{X} \subseteq \Modu (S)$ and $\mathcal{Y} \subseteq \Modu (S^{\rm op})$, recall from \cite[Def. 2.1]{WangYangZhu19} that an $S$-module $M$ is \emph{Gorenstein $(\mathcal{X,Y})$-flat} if there exists an exact complex
\[
X_\bullet = \cdots \to X_1 \xrightarrow{\partial_1} X_0 \xrightarrow{\partial_0} X^0 \xrightarrow{\partial^0} X^1 \xrightarrow{\partial^1} \cdots
\]
of modules in $\mathcal{X}$ such that $Y \otimes_S X_\bullet$ is exact for every $Y \in \mathcal{Y}$ and $M \simeq \Ker(\partial^0)$. The class of Gorenstein $(\mathcal{X,Y})$-flat $S$-modules will be denoted by $\mathcal{GF}_{(\mathcal{X,Y})}$. 

We now show the Gorenstein flat version of the previous two results.

\begin{proposition}[relative Gorenstein flatness under the 1${}^{\rm st}$ Foxby equivalence] \label{GF-Equiv}
If ${}_S C_R$ is faithfully semidualizing, then the following assertions hold for every $M \in \B_{C}(S)$ and $N \in \A _C (R)$.
\begin{enumerate}
\item $M \in \GF_{(\M^C(S), \N^C(S^{\rm op}))}$ if, and only if, $\Hom_S(C, M) \in \GF_{(\mathfrak{M,N})}$.

\item $N \in \GF_{(\mathfrak{M,N})}$ if, and only if, $C \otimes_R N \in \GF_{(\M^C(S), \N^C(S^{\rm op}))}$.
\end{enumerate}
\end{proposition}

\begin{proof}
Let us only show part (1), as part (2) is an immediate consequence of (1). For the ``only if'' part, let $M \in \GF_{(\M^C(S), \N^C(S^{\rm op}))}$. Then, $M \simeq \Ker(\partial^0)$ in an exact complex 
\[
X_\bullet = \cdots \to X_1 \xrightarrow{\partial_1} X_0 \xrightarrow{\partial_0} X^0 \xrightarrow{\partial^0} X^1 \xrightarrow{\partial^1} \cdots
\] 
where $X_k, X^k \in \M^C(S)$ for every $k \in \mathbb{Z}_{\geq 0}$, such that $N \otimes_S X_\bullet$ is exact for every $N \in \N^C(S^{\rm op})$. Note that since $M, X_k, X^k \in \B_C(S)$ for every $k \in \mathbb{Z}_{\geq 0}$, and $\B_C(S)$ is thick, $X_\bullet$ is an exact complex with cycles in $\B_C(S)$, and so $\Hom_S(C,X_\bullet)$ is an exact complex by $\mathsf{(b1)}$, and with components in $\mathfrak{M}$ by Proposition \ref{Equivalencia1}. Moreover, $\Hom_S(C,M) = \Ker(\Hom_S(C,\partial^0))$, and hence it suffices to show that $B \otimes_R \Hom_S(C,X_\bullet)$ is exact for every $B \in \mathfrak{N}$. For such $B$, we know that $\Hom _{R^{\op}}(C, B) \in \N^C(S^{\rm op})$. Then, by Lemma \ref{lem:HolmWhite_for_MN} (3) we have
\begin{align*}
B \otimes_R \Hom_S(C,X_\bullet) & \simeq [\Hom _{R^{\op}}(C, B)] \otimes_S [C \otimes_R \Hom_S(C,X_\bullet)] \\
& \simeq [\Hom _{R^{\op}}(C, B)] \otimes_S X_\bullet.
\end{align*}
Since $[\Hom _{R^{\op}}(C, B)] \otimes_S X_\bullet$ is exact, so is $B \otimes_R \Hom_S(C,X_\bullet)$.

For the ``if'' part, suppose that $\Hom_S(C, M) \in \GF_{(\mathfrak{M,N})}$, that is, $\Hom_S(C, M) \simeq \Ker(\partial^0)$ in an exact complex 
\[
A_\bullet = \cdots \to A_1 \xrightarrow{\partial_1} A_0 \xrightarrow{\partial_0} A^0 \xrightarrow{\partial^0} A^1 \xrightarrow{\partial^1} \cdots
\]
with components in $\mathfrak{M}$ and such that $B \otimes_R A_\bullet$ is exact for every $B \in \mathfrak{N}$. Note that $A_\bullet$ is an exact complex with cycles in $\A_C(R)$, and so $C \otimes_R A_\bullet$ is exact with components in $\M^C(S)$ and $M \simeq C \otimes_R \Hom_S(C, M) \simeq \Ker(C \otimes_R \partial^0)$. Proceeding as in the ``only if'' part, one can show that $N \otimes_S [C \otimes_R A_\bullet]$ is exact for every $N \in \mathcal{N}^C(S^{\rm op})$.
\end{proof}

The following is immediate.

\begin{corollary}[Gorenstein flat Foxby equivalence] \label{Gorenstein-flat-equivalence}
If $_S C _R$ is faithfully semidualizing, then there is an equivalence of categories
\[
\begin{tikzpicture}[description/.style={fill=white,inner sep=2pt}] 
\matrix (m) [ampersand replacement=\&, matrix of math nodes, row sep=6em, column sep=5em, text height=1.25ex, text depth=0.25ex] 
{ 
\A_C(R) \cap \GF_{(\mathfrak{M,N})}  \&  \GF_{(\M^C(S), \N^C(S^{\rm op}))} \cap \B_C(S) \\
}; 
\path[->] 
(m-1-1) edge [bend left = 15] node[above] {\footnotesize$C \otimes_R \sim$} node[below] {\footnotesize$\sim$} (m-1-2)
(m-1-2) edge [bend left = 15] node[above] {\footnotesize$\sim$} node[below] {\footnotesize$\Hom_S(C,\sim)$} (m-1-1)
;
\end{tikzpicture} 
\]
\end{corollary}

Given $\X \subseteq \Modu (S)$ and $\Y \subseteq \Modu (S^{\rm op})$, in \cite{WangYangZhu19} the authors find some sufficient conditions under which $\GF_{(\X,\Y)}$ is closed under extensions and forms the class of cofibrant objects of an abelian model structure. Namely, if $(\X,\Y)$ is a Tor-orthogonal complete duality pair over $S$ with $\X$ resolving, then $\GF_{(\X,\Y)}$ is resolving and there exists a hereditary abelian model structure on $\Modu (S)$ such that $\GF_{(\X,\Y)}$ (resp., $\X$) is the class of (resp., trivially) cofibrant objects, and $\X^\perp$ (resp., $[\GF_{(\X,\Y)}]^\perp$) is the class of (resp., trivially) fibrant objects. We obtain the following result concerning homotopical aspects of $\GF_{(\M^C(S),\N^C(S^{\rm op}))}$.

\begin{corollary}[homotopical aspects of relative Gorenstein flats] \label{coro:first_model_structure}
If ${}_S C_R$ is faithfully semidualizing and $(\M,\N)$ is a complete Tor-orthogonal duality pair over $R$ with $\M$ resolving and $S \in \M^C(S)$, then $\GF_{(\M^C(S),\N^C(S^{\rm op}))}$ is resolving, and there exists a hereditary abelian model structure on $\Modu (S)$ such that $\GF_{(\M^C(S),\N^C(S^{\rm op}))}$ (resp., $\M^C(S)$) is the class of (trivially) cofibrant objects, and $[\M^C(S)]^\perp$ (resp., $[\GF_{(\M^C(S),\N^C(S^{\rm op}))}]^\perp$) is the class of (trivially) fibrant objects. Furthermore, its homotopy category is triangle equivalent to the stable category 
\[
[\GF_{(\M^C(S),\N^C(S^{\rm op}))} \cap [\M^C(S)]^\perp] / \sim,
\] 
where $[\GF_{(\M^C(S),\N^C(S^{\rm op}))} \cap [\M^C(S)]^\perp$ is a Frobenius category in which $\M^C(S) \cap [\M^C(S)]^\perp$ is the class of projective-injective objects, and for $f, g \colon M \to N$ one has that $f \sim g$ if $f - g$ factors through an object in $\M^C(S) \cap [\M^C(S)]^\perp$.  
\end{corollary}

\begin{proof}
First, note that $\M \cap \A_C(R)$ is resolving, since $\A_C(R)$. Then by Proposition \ref{Prop0} (5) and (6), and the properties of perfect duality pairs, we have that $(\M^C(S),\N^C(S^{\rm op}))$ is a complete duality pair over $S$ with $\M^C(S)$ resolving, which is also Tor-orthogonal by Proposition \ref{prop:Tor-orthogonal}. The existence of the mentioned model structure follows by the previous comments on \cite{WangYangZhu19}. Finally, the description of its homotopy category is a consequence of \cite[Prop. 2.2 and Thm. 4.3]{GillespieHereditary}. 
\end{proof}

We now explore the relation between relative Gorenstein injectives and Gorenstein flats via Pontryagin duality.

\begin{proposition}[Gorenstein duality relations]\label{Dual-P}
The following assertions hold for every $M \in \Modu (S)$ and $N \in \Modu (R)$. 
\begin{enumerate}
\item If $M \in  \GF_{(\M^C(S), \N^C(S^{\rm op}))}$, then $M^+ \in \GI_{\N^C(S^{\rm op})}$.

\item If $N \in \GF_{(\mathfrak{M,N})}$, then $N^+ \in \GI_{\mathfrak{N}}$.
\end{enumerate}
\end{proposition}

\begin{proof} 
For part (1), let $M \in  \GF_{(\M^{C}(S), \N^{C}(S^{\rm op}))}$. Then, $M \simeq \Ker(\partial^0)$ in an exact complex 
\[
X_\bullet = \cdots \to X_1 \xrightarrow{\partial_1} X_0 \xrightarrow{\partial_0} X^0 \xrightarrow{\partial^0} X^1 \xrightarrow{\partial^1} \cdots
\] 
where $X_k, X^k \in \M^C(S)$ for every $k \in \mathbb{Z}_{\geq 0}$, such that $N \otimes_S X_\bullet$ is exact for every $N \in \N^C(S^{\rm op})$. By Proposition \ref{Dualidad1}, 
\[
X_\bullet^+ = \cdots \xrightarrow{(\partial^1)^+} (X^1)^+ \xrightarrow{(\partial^0)^+} (X^0)^+ \xrightarrow{(\partial_0)^+} (X_0)^+ \xrightarrow{(\partial_1)^+} (X_1)^+ \to \cdots
\]
is an exact complex with components in $\mathcal{N}^C(S^{\rm op})$ with $M^+ \simeq \Ker((\partial_0)^+)$. It remains to show that $\Hom_{S^{\rm op}}(N,X_\bullet^+)$ is exact for every $N \in \mathcal{N}^C(S^{\rm op})$. For,
\begin{align*}
\Hom_{S^{\rm op}}(N,X_\bullet^+) & \simeq \Hom_{S^{\rm op}}(\Hom_{R^{\rm op}}(C,N \otimes_S C),\Hom_{R^{\rm op}}(C, X^+_\bullet \otimes_S C)) & \text{[2nd. Foxby eq.]} \\
& \simeq \Hom_{R^{\rm op}}(N \otimes_S C, X^+_\bullet \otimes_S C) & \text{[Lemma \ref{lem:HolmWhite_for_MN} (2)]} \\
& \simeq \Hom_{R^{\rm op}}(N \otimes_S C, [\Hom_S(C,X_\bullet)]^+) & \text{\cite[Thm. 3.2.11]{EnochsJenda00}} \\
& \simeq [(N \otimes_S C) \otimes_R \Hom_S(C,X_\bullet) ]^+ & \text{\cite[Thm. 2.1.10]{EnochsJenda00}} \\
& \simeq [N \otimes_S X_\bullet]^+ & \text{\cite[Thm. 6.4 (c)]{HolmWhite07}}
\end{align*}
where $N \otimes_S X_\bullet$ is exact, and so is $[N \otimes_S X_\bullet]^+$. Hence $\Hom_{S^{\rm op}}(N,X_\bullet^+)$ is exact. 

Now for part (2), we know from part (1) that $[\GF _{(\M^C(S), \N^C(S^{\rm op}))}]^{+} \subseteq \GI_{\N^C(S^{\rm op})}$. On the other hand, $R$ is a (faithfully) semidualizing $(R,R)$-bimodule, and  $\A _R(R) = \Modu (R)$ and $\B_R(R^{\op}) = \Modu (R ^{\op})$. It follows that for any duality pair $(\tilde{\M},\tilde{\N})$ over $R$ we have $\tilde{\M}^R(R) = \tilde{\M} \cap \A_R(R) = \tilde{\M}$ and $\tilde{\N}^R(R^{\op}) = \tilde{\N} \cap \B_R(R^{\rm op}) = \tilde{\N}$. The previous containment implies $[\GF_{(\M,\N)}]^{+} \subseteq \GI_{\N}$. In particular, since $(\mathfrak{M,N})$ is a duality pair over $R$ by \cite[Prop. 2.8]{HJ09}, we get $[\GF_{(\mathfrak{M,N})}]^{+} \subseteq \GI_{\mathfrak{N}}$.
\end{proof}

The following is a consequence of Corollaries \ref{Gorenstein-injective-equivalence}, \ref{Gorenstein-flat-equivalence}, Proposition \ref{Dual-P}, and the fact that $(\A_C(R),\B_C(R^{\rm op}))$ and $(\A_C(S^{\rm op}),\B_C(S))$ are symmetric duality pairs.

\begin{corollary}[Gorenstein flat Foxby equivalences and duality]
If $_S C _R$ is faithfully semidualizing, then in the following diagram, the horizontal arrows are equivalences: 
\[
\begin{tikzpicture}[description/.style={fill=white,inner sep=2pt}] 
\matrix (m) [ampersand replacement=\&, matrix of math nodes, row sep=8em, column sep=5em, text height=1.25ex, text depth=0.25ex] 
{ 
\A_C(R) \cap \GF_{(\mathfrak{M,N})} \& \GF_{(\M^C(S),\N^C(S^{\rm op}))} \cap \B_C(S) \\
\B_C(R^{\rm op}) \cap \GI_{\mathfrak{N}} \& \GI_{\N^C(S^{\rm op})} \cap \A_C(S^{\rm op}) \\
}; 
\path[->] 
(m-1-1) edge [bend left = 15] node[above] {\footnotesize$C \otimes_R \sim$} node[below] {\footnotesize$\sim$} (m-1-2)
(m-1-2) edge [bend left = 15] node[above] {\footnotesize$\sim$} node[below] {\footnotesize$\Hom_S(C,\sim)$} (m-1-1)
(m-2-2) edge [bend right = 15] node[above] {\footnotesize$- \otimes_S C$} node[below] {\footnotesize$\sim$} (m-2-1)
(m-2-1) edge [bend right = 15] node[above] {\footnotesize$\sim$} node[below] {\footnotesize$\Hom_{R^{\rm op}}(C,\sim)$} (m-2-2)
(m-1-1) edge node[left] {\footnotesize$(-)^+$} (m-2-1)
(m-1-2) edge node[right] {\footnotesize$(-)^+$} (m-2-2)
;
\end{tikzpicture} 
\]
\end{corollary}

A natural question at this point is whether it is possible to complete the previous diagram with 
\[
(-)^+ \colon \B_C(R^{\rm op}) \cap \GI_{\mathfrak{N}} \to \A_C(R) \cap \GF_{(\mathfrak{M,N})}
\]
and
\[
(-)^+ \colon \GI_{\N^C(S^{\rm op})} \cap \A_C(S^{\rm op}) \to \GF_{(\M^C(S),\N^C(S^{\rm op}))} \cap \B_C(S).
\] 
We are not aware if this is the case in general, but we can at least improve the duality relations in Proposition \ref{Dual-P} by modifying slightly the previous classes of relative Gorenstein injectives and Gorenstein flats. Specifically, we shall need to consider kernels of certain classes of modules.

Let us present the following notation motivated by \cite[\S 2.7]{ChengZhao24}:
\begin{align*}
\mathcal{H}(\mathfrak{N}) & := {}^\perp\mathfrak{N} \cap \mathfrak{N}, \\ 
\mathcal{H}(\mathcal{N}^C(S^{\rm op})) & = {}^\perp[\mathcal{N}^C(S^{\rm op})] \cap \mathcal{N}^C(S^{\rm op}), \\
\mathcal{H}_C(\mathfrak{N}) & := \{ M \in \Modu (S) \ \text{ : } \ M \simeq \Hom_{R^{\rm op}}(C,N) \text{ for some } N \in \mathcal{H}(\mathfrak{N}) \}.
\end{align*}

The following two results characterizes the total orthogonal complement ${}^\perp\mathcal{N}^C(S^{\rm op})$, and relate the classes $\mathcal{H}(\mathcal{N}^C(S^{\rm op}))$, $\mathcal{H}_C(\mathfrak{N})$ and $\mathcal{H}(\mathfrak{N})$. These are relative versions of \cite[Lem. 3.3 \& Prop. 3.4]{ChengZhao24}, and can be proved using similar arguments.

\begin{proposition}[cores of $\mathfrak{N}$ and $\mathcal{N}^C(S^{\rm op})$] \label{raro}
The following assertions hold:
\begin{enumerate}
\item The following are equivalent for every $M \in \Modu (S^{\rm op})$:
\begin{enumerate}[(a)]
\item $M \in {}^\perp[\mathcal{N}^C(S^{\rm op})]$.
\item $M \otimes_S C \in {}^\perp\mathfrak{N}$ and $\Tor^S_i(M,C) = 0$ for every $i \in \mathbb{Z}_{> 0}$.
\end{enumerate}

\item The following are equivalent for every $N \in \Modu (S^{\rm op})$:
\begin{enumerate}[(a)]
\item $N \in \mathcal{H}(\mathcal{N}^C(S^{\rm op}))$.
\item $N \in \mathcal{H}_C(\mathfrak{N})$.
\item $N \otimes_S C \in \mathcal{H}(\mathfrak{N})$.
\end{enumerate}
\end{enumerate}
\end{proposition}

The following is immediate from the definition of $\mathcal{H}_C (\mathfrak{N})$ and the previous proposition.

\begin{corollary}[Foxby equivalences for cores]
In the following diagram, the horizontal arrows are equivalences and the vertical arrows are inclusions of categories: 
\[
\begin{tikzpicture}[description/.style={fill=white,inner sep=2pt}] 
\matrix (m) [ampersand replacement=\&, matrix of math nodes, row sep=7em, column sep=4.5em, text height=1.25ex, text depth=0.25ex] 
{ 
\mathcal{H}_C(\mathfrak{N}) \& \mathcal{H}(\mathfrak{N}) \\
\A_C(S^{\rm op}) \& \B_C(R^{\rm op}) \\
}; 
\path[->] 
(m-1-1) edge [bend left = 20] node[above] {\footnotesize$- \otimes_S C$} node[below] {\footnotesize$\sim$} (m-1-2)
(m-1-2) edge [bend left = 20] node[above] {\footnotesize$\sim$} node[below] {\footnotesize$\Hom_{R^{\rm op}}(C,\sim)$} (m-1-1)
(m-2-1) edge [bend left = 20] node[above] {\footnotesize$- \otimes_S C$} node[below] {\footnotesize$\sim$} (m-2-2)
(m-2-2) edge [bend left = 20] node[above] {\footnotesize$\sim$} node[below] {\footnotesize$\Hom_{R^{\rm op}}(C,\sim)$} (m-2-1)
;
\path[>->]
(m-1-1) edge (m-2-1) 
(m-1-2) edge (m-2-2) 
;
\end{tikzpicture} 
\]
\end{corollary}

As mentioned before, $(\mathfrak{M,N})$ is a duality pair over $R$. The following relations can be proved as in Propositions \ref{duality_Ginj} and \ref{GF-Equiv}, and using Proposition \ref{raro}. \newpage

\begin{proposition}[Gorenstein injective and flat Foxby equivalences for cores]\label{prop:Foxby-corazon}
Let ${}_S C_R$ be faithfully semidualizing. If $M \in \B_{C}(R^{\rm op})$ and $N \in \A_C(S^{\rm op})$, then:
\begin{enumerate}
\item $N \in \GI_{(\mathcal{H}_C(\mathfrak{N}),\N^C(S^{\rm op}))}$ if, and only if $N \otimes_S C \in \GI_{(\mathcal{H}(\mathfrak{N}),\mathfrak{N})}$.

\item $M \in \GI_{(\mathcal{H}(\mathfrak{N}),\mathfrak{N})}$ if, and only if, $\Hom_{R^{\op}}(C,M) \in \GI_{(\mathcal{H}_C(\mathfrak{N}),\N^C(S^{\rm op}))}$.
\end{enumerate} 
If $N \in \B_C(S)$ and $M \in \A_C(R)$, then:
\begin{enumerate}
\setcounter{enumi}{2}
\item $N \in \GF_{(\M^C(S), \mathcal{H}_C(\mathfrak{N}))}$ if, and only if, $\Hom_S(C, N) \in \GF_{(\mathfrak{M}, \mathcal{H}(\mathfrak{N}))}$.

\item $M \in \GF_{(\mathfrak{M}, \mathcal{H}(\mathfrak{N}))}$ if, and only if, $C \otimes_R M \in \GF_{(\M ^C(S), \mathcal{H}_C(\mathfrak{N}))}$.
\end{enumerate}
\end{proposition}

Under certain conditions, the duality relations in Proposition \ref{Dual-P} can be improved. Namely, if $(\mathfrak{M,N})$ is product closed bicomplete, then 
\[
(\GF_{(\mathfrak{M},\mathcal{H}(\mathfrak{N}))},\GI_{(\mathcal{H}(\mathfrak{N}),\mathfrak{N})})
\] 
is a perfect duality pair. On the other hand, if we assume that the pairs $(\M,\N)$ and $(\A_C(R),\B_C(R^{\rm op}))$ are bicomplete and that $\M$ closed under products, then by \cite[Ex. 3.10 (5)]{BecerrilPerez25} $(\mathfrak{M,N})$ is bicomplete, and $\mathfrak{M}$ is closed under products since $\A_C(R)$ has this closure property by \cite[Thm. 3.3 (1)]{Huang21}\footnote{In the cited reference, the author assumes that ${}_R C_S$ is semidualizing and shows that $\A_C(R^{\rm op})$ is closed under products. In our situation, keep in mind that ${}_S C_R$ is semidualizing if, and only if, ${}_{R^{\rm op}} C_{S^{\rm op}}$ is semidualizing.}. The pair $(\A_C(R),\B_C(R^{\rm op}))$ is bicomplete (or equivalently, $({}^\perp[\B_C(R^{\rm op})],\B_C(R^{\rm op}))$ is a cotorsion pair cogenerated by a set) for instance when $C_R$ is self-small and w-tilting (see \cite[Coroll. 3.12]{BDGO21}). Keep in mind that $\B_C(R^{\rm op})$ is always coresolving, and so the pair $({}^\perp[\B_C(R^{\rm op})],\B_C(R^{\rm op}))$  is indeed hereditary. We have the following result by \cite[Coroll. 3.14]{BecerrilPerez25}.

\begin{proposition}[Gorenstein duality relations for cores]\label{prop:dualityR}
If $({}^\perp[\B_C(R^{\rm op})],\B_C(R^{\rm op}))$ is a cotorsion pair cogenerated by a set, and $(\M,\N)$ is bicomplete with $\M$ closed under products, then $\GF_{(\mathfrak{M},\mathcal{H}(\mathfrak{N}))}$ is resolving and $(\GF_{(\mathfrak{M},\mathcal{H}(\mathfrak{N}))},\GI_{(\mathcal{H}(\mathfrak{N}),\mathfrak{N})})$ is a perfect duality pair. In particular, the following holds for every $N \in \Modu (R)$:
\begin{itemize}
\item $N \in \GF_{(\mathfrak{M},\mathcal{H}(\mathfrak{N}))}$ if, and only if, $N^+ \in \GI_{(\mathcal{H}(\mathfrak{N}),\mathfrak{N})}$.
\end{itemize}
\end{proposition}

Regarding homotopical aspects, we obtain the following model structure by \cite[Thm. 5.3]{BecerrilPerez25}.

\begin{proposition}[homotopical aspects of relative Gorenstein flats respect to cores]\label{GFmodelST}
If $({}^\perp[\B_C(R^{\rm op})],\B_C(R^{\rm op}))$ is a cotorsion pair cogenerated by a set, and $(\M,\N)$ is bicomplete with $\M$ closed under products, then there exists a unique hereditary abelian model structure on $\Modu (R)$ such that $\GF_{(\mathfrak{M},\mathcal{H}(\mathfrak{N}))}$ (resp., $\mathfrak{M}$) is the class of (trivially) cofibrant objects, and $\mathfrak{M}^\perp$ (resp., $[\GF_{(\mathfrak{M},\mathcal{H}(\mathfrak{N}))}]^\perp$) is the class of (trivially) fibrant objects. Moreover, its homotopy category is triangle equivalent to the stable category 
\[
[\GF_{(\mathfrak{M},\mathcal{H}(\mathfrak{N}))} \cap \mathfrak{M}^\perp] / \sim,
\]
where for $f, g \colon M \to N$ one has that $f \sim g$ if $f - g$ factors through an object in $\mathfrak{M} \cap \mathfrak{M}^\perp$.
\end{proposition}

Regarding the counterparts $\GF_{(\mathcal{M}^C(S),\mathcal{H}_C(\mathfrak{N}))}$ and $\GI_{(\mathcal{H}_C(\mathfrak{N}),\mathcal{N}^C(S^{\rm op}))}$, we look for conditions under which there is a duality relation between these classes. First, we note the following.

\begin{lemma}\label{lema-corazon}
The equality $\mathcal{H}(\mathcal{N}^C(S^{\rm op})) = \mathcal{H}_C(\mathfrak{N})$ holds.
\end{lemma}

\begin{proof}
First, let $N \in \mathcal{H}(\mathcal{N}^C(S^{\rm op}))$. Then, $N \simeq \Hom_{R^{\rm op}}(C,B)$ with $B \in \mathfrak{N}$ and $N \in {}^\perp[\mathcal{N}^C(S^{\rm op})]$. We can note that $B \in \mathcal{H}(\mathfrak{N})$. Indeed, let $B' \in \mathfrak{N}$. Then by Lemma \ref{lem:HolmWhite_for_MN} (2), we have
\[
\Ext^i_{R^{\rm op}}(B,B') \simeq \Ext^i_{S^{\rm op}}(N,\Hom_{R^{\rm op}}(C,B')) = 0
\]
since $\Hom_{R^{\rm op}}(C,B') \in \mathcal{N}^C(S^{\rm op})$. Hence, $B \in \mathcal{H}(\mathfrak{N})$ and so $N \in \mathcal{H}_C(\mathfrak{N})$.

The other containment can be proved similarly. 
\end{proof}

By the previous lemma, note that $\GF_{(\mathcal{M}^C(S),\mathcal{H}_C(\mathfrak{N}))} = \GF_{(\mathcal{M}^C(S),\mathcal{H}(\mathcal{N}^C(S^{\rm op})))}$ and $\GI_{(\mathcal{H}_C(\mathfrak{N}),\mathcal{N}^C(S^{\rm op}))} = \GI_{(\mathcal{H}(\mathcal{N}^C(S^{\rm op})),\mathcal{N}^C(S^{\rm op}))}$. Thus, $\GF_{(\mathcal{M}^C(S),\mathcal{H}_C(\mathfrak{N}))}$ and $\GI_{(\mathcal{H}_C(\mathfrak{N}),\mathcal{N}^C(S^{\rm op}))}$ are the relative Gorenstein flat and relative Gorenstein injective modules covered in \cite{BecerrilPerez25}, induced by the duality pair $(\M^C(S),\N^C(S^{\rm op}))$. However, we are not aware under which condition the previous duality pair is bicomplete, so we cannot obtain the duality relation described in Proposition \ref{prop:dualityR} via bicompleteness of $(\M^C(S),\N^C(S^{\rm op}))$. There is one way to overcome this limitation via Foxby equivalences if duality is restricted to the Auslander and Bass classes.

\begin{proposition}\label{otra_dualidad}
If $M \in \GF_{(\M^C(S),\mathcal{H}_C(\mathfrak{N}))}$, then $M^+ \in \GI_{(\mathcal{H}_C(\mathfrak{N}),\mathcal{N}^C(S^{\rm op}))}$. Furthermore, the converse holds for every $M \in \B_C(S)$ if in addition ${}_S C_R$ is faithfully semidualizing, $({}^\perp[\B_C(R^{\rm op})],\B_C(R^{\rm op}))$ is a cotorsion pair cogenerated by a set, and $(\M,\N)$ is bicomplete with $\M$ closed under products.
\end{proposition}

\begin{proof}
The first assertion follows as Proposition \ref{Dual-P} (1). In order to show the converse, let $M \in \B_C(S)$ such that $M^+ \in \GI_{(\mathcal{H}_C(\mathfrak{N}),\mathcal{N}^C(S^{\rm op}))}$. Then by Proposition \ref{prop:Foxby-corazon} (1) we have that $M^+ \simeq \Hom_{R^{\rm op}}(C,B)$ for some $B \in \B_C(R^{\rm op}) \cap \GI_{(\mathcal{H}(\mathfrak{N}),\mathfrak{N})}$. On the other hand, using the first Foxby equivalence, we can write $M \simeq C \otimes_R A$ with $A \in \A_C(R)$. Then, $M^+ \simeq (C \otimes_R A)^+ \simeq \Hom_{R^{\rm op}}(C,A^+)$. Since
\[
A^+ \simeq \Hom_{R^{\rm op}}(C,A^+) \otimes_S C \simeq \Hom_{R^{\rm op}}(C,B) \otimes_S C \simeq B \in B_C(R^{\rm op}) \cap \GI_{(\mathcal{H}(\mathfrak{N}),\mathfrak{N})},
\]
we have by Proposition \ref{prop:dualityR} that $A \in \GF_{(\mathfrak{M},\mathcal{H}(\mathfrak{N}))}$, and hence $M \simeq C \otimes_R A \in \GF_{(\M^C(S),\mathcal{H}_C(\mathfrak{N}))}$ by Proposition \ref{prop:Foxby-corazon} (4). 
\end{proof}

\begin{corollary}[Gorenstein flat Foxby equivalences and duality for cores]\label{relationGFvsGI}
If ${}_S C_R$ is faithfully semidualizing, $({}^\perp[\B_C(R^{\rm op})],\B_C(R^{\rm op}))$ is a cotorsion pair cogenerated by a set, and $(\M,\N)$ is bicomplete with $\M$ closed under products, then we have the following diagram where the horizontal arrows are equivalences: 
\[
\begin{tikzpicture}[description/.style={fill=white,inner sep=2pt}] 
\matrix (m) [ampersand replacement=\&, matrix of math nodes, row sep=8em, column sep=3.5em, text height=1.25ex, text depth=0.25ex] 
{ 
\A_C(R) \cap \GF_{(\mathfrak{M}, \He(\mathfrak{N}))}  \& \B_C(S) \cap \GF_{(\mathcal{M}^C(S), \He_C(\mathfrak{N}))} \\
B_C(R^{\op}) \cap \GI_{(\He(\mathfrak{N}), \mathfrak{N})} \& \A_C(S^{\op}) \cap \GI_{(\He_C(\mathfrak{N}), \mathcal{N}^C(S^{\rm op}))}  \\
}; 
\path[->] 
(m-1-1) edge [bend left = 15] node[above] {\footnotesize$C \otimes_R \sim$} node[below] {\footnotesize$\sim$} (m-1-2)
(m-1-2) edge [bend left = 15] node[above] {\footnotesize$\sim$} node[below] {\footnotesize$\Hom_S(C,\sim)$} (m-1-1)
(m-2-2) edge [bend right = 15] node[above] {\footnotesize$- \otimes_S C$} node[below] {\footnotesize$\sim$} (m-2-1)
(m-2-1) edge [bend right = 15] node[above] {\footnotesize$\sim$} node[below] {\footnotesize$\Hom_{R^{\rm op}}(C,\sim)$} (m-2-2)
(m-1-1) edge node[left] {\footnotesize$(-)^+$} (m-2-1)
(m-1-2) edge node[right] {\footnotesize$(-)^+$} (m-2-2)
;
\end{tikzpicture} 
\]
\end{corollary}

Our last versions of relative Gorenstein injective and relative Gorenstein flat come precisely from the definition of Gorenstein $G_C\text{-FP}_n$-injective mentioned at the beginning of this section. Our approach to these versions consider dual Cohen-Macaulay objects, defined by Beligiannis and Reiten in \cite[Ch. V, \S 4]{BR07}. Given an object $M$ and a \emph{self-orthogonal} class $\omega$ of objects of an abelian category $\mathcal{C}$, that is $\Ext^{\geq 1}_{\mathcal{C}}(\omega,\omega) = 0$, one says that $M$ is \emph{dual Cohen-Macaulay} relative to $\omega$ if  $M \in \omega^\perp$ and there exists an $\omega$-resolution 
\[
\cdots \to W_1 \xrightarrow{\partial_1} W_0 \xrightarrow{\partial_0} M \to 0
\]
such that $\Ker(\partial_k) \in \omega^\perp$ for every $k \in \mathbb{Z}_{\geq 0}$.

\begin{lemma}\label{orthogonal_core}
The following assertions hold:
\begin{enumerate}
\item $\Ext^{\geq 1}_{S^{\rm op}}(\mathcal{N}^C(S^{\rm op}), \mathcal{H}_C(\mathfrak{N})) = 0$. In particular, $\mathcal{H}_C(\mathfrak{N})$ is self-orthogonal.
\item There is an equivalence of categories 
\[
\begin{tikzpicture}[description/.style={fill=white,inner sep=2pt}] 
\matrix (m) [ampersand replacement=\&, matrix of math nodes, row sep=5em, column sep=4.5em, text height=1.25ex, text depth=0.25ex] 
{ 
[\mathcal{H}_C(\mathfrak{N})]^\perp \cap \A_C(S^{\rm op}) \& \B_C(R^{\rm op}) \cap [\mathcal{H}(\mathfrak{N})]^\perp \\
}; 
\path[->] 
(m-1-1) edge [bend left = 15] node[above] {\footnotesize$- \otimes_S C$} node[below] {\footnotesize$\sim$} (m-1-2)
(m-1-2) edge [bend left = 15] node[above] {\footnotesize$\sim$} node[below] {\footnotesize$\Hom_{R^{\rm op}}(C,\sim)$} (m-1-1)
;
\end{tikzpicture} 
\]
\item If in addition $\mathcal{N}$ is closed under extensions, then so is $\mathcal{H}_C(\mathfrak{N})$.
\end{enumerate}
\end{lemma}

\begin{proof}
Parts (1) and (3) follow by Proposition \ref{Prop0} (1) and Lemma \ref{lema-corazon}. Part (2), on the other hand, follows by Lemma \ref{lem:HolmWhite_for_MN}.
\end{proof}

If we set $\omega = \mathcal{H}_C(\mathfrak{N})$, then following the notation in \cite{BR07}, we let $\mathrm{coCM}(\mathcal{H}_C(\mathfrak{N}))$ denote the class of dual Cohen-Macaulay $S^{\rm op}$-modules relative to $\mathcal{H}_C(\mathfrak{N})$. Let us present the Cohen-Macaulay analogue of Proposition \ref{duality_Ginj}.

\begin{proposition}[Foxby equivalences for dual Cohen-Macaulay modules]\label{duality_coCM}
If ${}_S C_R$ is faithfully semidualizing, then the following assertions hold for every $N \in \B_{C}(R^{\rm op})$ and $M \in \A_C(S^{\rm op})$:
\begin{enumerate}
\item $M \in \mathrm{coCM}(\mathcal{H}_C(\mathfrak{N}))$ if, and only if $M \otimes_S C \in \mathrm{coCM}(\mathcal{H}(\mathfrak{N}))$.

\item $N \in \mathrm{coCM}(\mathcal{H}(\mathfrak{N}))$ if, and only if, $\Hom_{R^{\op}}(C,N) \in \mathrm{coCM}(\mathcal{H}_C(\mathfrak{N}))$.
\end{enumerate} 
\end{proposition}

\begin{proof}
Follows from the arguments shown in the proof of Proposition \ref{duality_Ginj}, along with part (2) of the previous lemma. 
\end{proof}

The following result is a generalization of \cite[Prop. 3.7 \& Thm. 3.9]{ChengZhao24}.

\begin{proposition}[characterization of $\mathrm{coCM}(\mathcal{H}_C(\mathfrak{N}))$]\label{char_coCMinj}
Consider the following assertions for $M \in \Modu (S^{\rm op})$:
\begin{enumerate}[(a)]
\item $M \in \mathrm{coCM}(\mathcal{H}_C(\mathfrak{N}))$.

\item $M \simeq \Ker(\partial^0)$ in some exact complex
\[
\Gamma \colon \cdots \to N_1 \to N_0 \xrightarrow{\partial_0} I^0 \xrightarrow{\partial^0} I^1 \to \cdots
\]
where $N_k \in \mathcal{N}^C(S^{\rm op})$ and $I^k \in \Inj(S^{\rm op})$ for every $k \in \mathbb{Z}_{\geq 0}$, such that $\Hom_{S^{\rm op}}(H,\Gamma)$ is exact for every $H \in \mathcal{H}_C(\mathfrak{N})$. 

\item There exists a short exact sequence 
\[
0 \to L \to N \to M \to 0
\] 
with $N \in \mathcal{N}^C(S^{\rm op})$ and $L \in \mathrm{coCM}(\mathcal{H}_C(\mathfrak{N}))$.
\end{enumerate}
Then, the implications (a) $\Rightarrow$ (c) $\Rightarrow$ (b) always hold. If in addition $(\mathcal{M,N})$ is bicomplete and $({}^\perp[\B_C(R^{\rm op})],\B_C(R^{\rm op}))$ is a cotorsion pair cogenerated by a set, then (b) $\Rightarrow$ (a) holds and the three assertions are equivalent. In particular, $M \in \mathrm{coCM}(\mathcal{H}_C(\mathfrak{N}))$ if, and only if, $M \in [\mathcal{H}_C(\mathfrak{N})]^\perp$ and there exists an $\mathcal{N}^C(S^{\rm op})$-resolution 
\[
\cdots \to N_1 \xrightarrow{\partial_1} N_0 \xrightarrow{\partial_0} M \to 0
\]
such that $\Ker(\partial_k) \in [\mathcal{H}_C(\mathfrak{N})]^\perp$ for every $k \in \mathbb{Z}_{\geq 0}$. That is, 
\[
\mathrm{coCM}(\mathcal{H}_C(\mathfrak{N})) = \mathcal{WGI}_{(\mathcal{H}_C(\mathfrak{N}),\mathcal{N}^C(S^{\rm op}))}
\] 
according to the notation of \cite{BMS}. 
\end{proposition}

\begin{proof}
The implication (a) $\Rightarrow$ (c) is immediate. For (c) $\Rightarrow$ (b), consider a $\mathcal{H}_C(\mathfrak{N})$-resolution $\cdots \to W_1 \xrightarrow{\partial_1} W_0 \xrightarrow{\partial_0} L \to 0$ with $\text{Ker}(\partial_k) \in [\mathcal{H}_C(\mathfrak{N})]^\perp$ for every $k \in \mathbb{Z}_{\geq 0}$. Since also $L \in [\mathcal{H}_C(\mathfrak{N})]^\perp$, this resolution remains exact after applying the functor $\Hom_{S^{\rm op}}(H,\sim)$ for every $H \in \mathcal{H}_C(\mathfrak{N})$. Moreover, the short exact sequence $0 \to L \to N \to M \to 0$ remains exact as well after applying $\Hom_{S^{\rm op}}(H,\sim)$ (again, since $L \in [\mathcal{H}_C(\mathfrak{N})]^\perp$). On the other hand, by Lemma \ref{orthogonal_core} (1), we have that $N \in [\mathcal{H}_C(\mathfrak{N})]^\perp$, and since $[\mathcal{H}_C(\mathfrak{N})]^\perp$ is coresolving, we get $M \in [\mathcal{H}_C(\mathfrak{N})]^\perp$. Now consider an injective coresolution $0 \to M \to I^0 \to I^1 \to \cdots$ of $M$. Since $M \in [\mathcal{H}_C(\mathfrak{N})]^\perp$, $\Inj(S^{\rm op}) \subseteq [\mathcal{H}_C(\mathfrak{N})]^\perp$ and $[\mathcal{H}_C(\mathfrak{N})]^\perp$ is coresolving, we have that the previous coresolution has cycles in $[\mathcal{H}_C(\mathfrak{N})]^\perp$, and so it remains exact after applying $\Hom_{S^{\rm op}}(H,\sim)$ for every $H \in \mathcal{H}_C(\mathfrak{N})$. Glueing the $\mathcal{H}_C(\mathfrak{N})$-resolution of $L$ and the injective coresolution of $M$, we obtain the desired complex.     

Finally, for (b) $\Rightarrow$ (a), suppose we are given a complex as $\Gamma$. Using a standard homological algebra argument and Lemma \ref{orthogonal_core} (1), we can note that $\Gamma$ has cycles in $[\mathcal{H}_C(\mathfrak{N})]^\perp$. From this point on, we follow the arguments in \cite[Proof of Thm. 3.9 (1) $\Rightarrow$ (3)]{ChengZhao24}. Consider the short exact sequence $0 \to \Ker(\partial_0) \to N_0 \to M \to 0$, where $N_0 \simeq \Hom_{R^{\rm op}}(C,B_0)$ with $B_0 \in \mathfrak{N}$. By the comments prior to Proposition \ref{prop:dualityR}, $(\mathfrak{M,N})$ is a bicomplete duality pair. In particular, $({}^\perp\mathfrak{N},\mathfrak{N})$ is a complete cotorsion pair, and so we can consider a short exact sequence $0 \to B'_0 \to H_0 \to B_0 \to 0$ with $B'_0 \in \mathfrak{N}$ and $H_0 \in \mathcal{H}(\mathfrak{N})$. Since $B'_0 \in \mathcal{B}_C(R^{\rm op})$, the sequence 
\[
0 \to \Hom_{R^{\rm op}}(C,B'_0) \to \Hom_{R^{\rm op}}(C,H_0) \to N_0 \to 0
\]
is exact. Now taking the pullback of $\Ker(\partial_0) \to N_0 \leftarrow \Hom_{R^{\rm op}}(C,H_0)$ yields the following commutative diagram with exact rows and columns
\[
\begin{tikzpicture}[description/.style={fill=white,inner sep=2pt}] 
\matrix (m) [matrix of math nodes, row sep=2.3em, column sep=2.3em, text height=1.25ex, text depth=0.25ex] 
{ 
{} & 0 & 0 & {} & {} \\
{} & \Hom_{R^{\rm op}}(C,B'_0) & \Hom_{R^{\rm op}}(C,B'_0) & {} & {} \\
0 & Q & \Hom_{R^{\rm op}}(C,H_0) & M & 0 \\
0 & \Ker(\partial_0) & N_0 & M & 0 \\
{} & 0 & 0 & {} & {} \\
}; 
\path[->] 
(m-3-2)-- node[pos=0.5] {\footnotesize$\mbox{\bf pb}$} (m-4-3) 
(m-1-2) edge (m-2-2) (m-1-3) edge (m-2-3)
(m-2-2) edge (m-3-2) (m-2-3) edge (m-3-3)
(m-3-2) edge (m-4-2) (m-3-3) edge (m-4-3)
(m-4-2) edge (m-5-2) (m-4-3) edge (m-5-3)
(m-3-1) edge (m-3-2) (m-3-2) edge (m-3-3) (m-3-3) edge (m-3-4) (m-3-4) edge (m-3-5)
(m-4-1) edge (m-4-2) (m-4-2) edge (m-4-3) (m-4-3) edge (m-4-4) (m-4-4) edge (m-4-5)
;
\path[-,font=\scriptsize]
(m-2-2) edge [double, thick, double distance=2pt] (m-2-3)
(m-3-4) edge [double, thick, double distance=2pt] (m-4-4)
;
\end{tikzpicture} 
\]
where $\Hom_{R^{\rm op}}(C,H_0) \in \mathcal{H}_C(\mathfrak{N})$. Note that $\Ker(\partial_0) \in \mathcal{WGI}_{(\mathcal{H}_C(\mathfrak{N}),\mathcal{N}^C(S^{\rm op}))}$. Also, $\Hom_{R^{\rm op}}(C,B'_0) \in \mathcal{N}^C(S^{\rm op})$, $\mathcal{H}_C(\mathfrak{N}) \subseteq {}^{\perp}\mathcal{N}^C(S^{\rm op})$ and $\mathcal{N}^C(S^{\rm op})$ is closed under extensions by Proposition \ref{Prop0} (1). Then by \cite[Dual of Lem. 3.23]{BMS} we have that $Q \in \mathcal{WGI}_{(\mathcal{H}_C(\mathfrak{N}),\mathcal{N}^C(S^{\rm op}))}$. Thus, we can iterate the previous procedure starting from $Q$ in order to obtain a $\mathcal{H}_C(\mathfrak{N})$-resolution of $M$ with cycles in $[\mathcal{H}_C(\mathfrak{N})]^\perp$, and hence proving that $M \in \mathrm{coCM}(\mathcal{H}_C(\mathfrak{N}))$. 
\end{proof}

We know relate $\mathrm{coCM}(\mathcal{H}_C(\mathfrak{N}))$ with the $(\mathcal{H}_C(\mathfrak{N}),\mathcal{N}^C(S^{\rm op}))$-Gorenstein injective modules. First, we recall from \cite[Def. 3.6]{BMS} that a pair $(\mathcal{Z,Y})$ of classes of $S^{\rm op}$-modules is \emph{GI-admissible} if:
\begin{itemize}
\item Every $S^{\rm op}$-module can be embedded into an $S^{\rm op}$-module of $\mathcal{Y}$; 
\item $\Ext^i_{S^{\rm op}}(Z,Y) = 0$ for every $Z \in \mathcal{Z}$, $Y \in \mathcal{Y}$ and $i \in \mathbb{Z}_{> 0}$;
\item $\mathcal{Z}$ and $\mathcal{Y}$ are closed under finite coproducts, and $\mathcal{Y}$ is closed under extensions; and
\item for every $Y \in \mathcal{Y}$ there exists a short exact sequence $0 \to Y' \to V \to Y \to 0$ with $Y' \in \mathcal{Y}$ and $V \in \mathcal{Y} \cap \mathcal{Z}$.
\end{itemize}

\begin{example}
If $({}^\perp[\B_C(R^{\rm op})],\B_C(R^{\rm op}))$ is a cotorsion pair cogenerated by a set, and $(\mathcal{M,N})$ is a bicomplete duality pair with $S \in \mathcal{M}^C(S)$, then the pair $(\mathcal{H}_C(\mathfrak{N}),\mathcal{N}^C(S^{\rm op}))$ is GI-admissible. Indeed, note first that $(\mathcal{M}^C(S),\mathcal{N}^C(S^{\rm op}))$ is a complete (in particular, perfect) duality pair by Proposition \ref{Prop0} (6), and so $\Inj(S^{\rm op}) \subseteq \mathcal{N}^C(S^{\rm op})$. It follows that every $S^{\rm op}$-module can be embedded into a(n injective) module in $\mathcal{N}^C(S^{\rm op})$. Since $\mathcal{H}_C(\mathfrak{N}) = \mathcal{N}^C(S^{\rm op}) \cap {}^\perp\mathcal{N}^C(S^{\rm op})$ by Lemma \ref{lema-corazon}, the second condition of the definition of GI-admissible pair is clear. Moreover, since $\mathcal{N}$ and $\mathcal{B}_C(R^{\rm op})$ are closed under finite direct sums and $\Hom_{R^{\rm op}}(C,\sim)$ commutes with finite direct sums, we have that $\mathcal{N}^C(S^{\rm op})$ has the same closure property, and clearly so does $\mathcal{H}_C(\mathfrak{N})$. Furthermore, since $\mathcal{M}^C(S)$ is closed under extensions and the duality pair $(\mathcal{M}^C(S),\mathcal{N}^C(S^{\rm op}))$ is symmetric, then $\mathcal{N}^C(S^{\rm op})$ is closed under extensions. Finally, by the same arguments used in the proof of (b) $\Rightarrow$ (a) of the previous proposition, every module in $\mathcal{N}^C(S^{\rm op})$ is the epimorphic image of a module in $\mathcal{H}_C(\mathfrak{N})$ with kernel in $\mathcal{N}^C(S^{\rm op})$.
\end{example}

\begin{proposition}\label{coCM_meets_GI}
If $({}^\perp[\B_C(R^{\rm op})],\B_C(R^{\rm op}))$ is a cotorsion pair cogenerated by a set, and $(\mathcal{M,N})$ is a bicomplete duality pair with $S \in \mathcal{M}^C(S)$, then 
\[
\mathrm{coCM}(\mathcal{H}_C(\mathfrak{N})) = \mathcal{GI}_{(\mathcal{H}_C(\mathfrak{N}),\mathcal{N}^C(S^{\rm op}))}.
\]
\end{proposition}

\begin{proof}
On the one hand, $\mathrm{coCM}(\mathcal{H}_C(\mathfrak{N})) = \mathcal{WGI}_{(\mathcal{H}_C(\mathfrak{N}),\mathcal{N}^C(S^{\rm op}))}$ by Proposition \ref{char_coCMinj}. On the other hand, $\mathcal{WGI}_{(\mathcal{H}_C(\mathfrak{N}),\mathcal{N}^C(S^{\rm op}))} = \mathcal{GI}_{(\mathcal{H}_C(\mathfrak{N}),\mathcal{N}^C(S^{\rm op}))}$ follows by \cite[Thm. 3.32]{BMS}, since the pair $(\mathcal{H}_C(\mathfrak{N}),\mathcal{N}^C(S^{\rm op}))$ is GI-admissible by the previous example.
\end{proof}

We know propose and study the Gorenstein flat counterpart of $\mathrm{coCM}(\mathcal{H}_C(\mathfrak{N}))$. In what follows, we shall consider the following notations:
\begin{align*}
\mathcal{H}(\mathfrak{M}) & := \mathfrak{M} \cap \mathfrak{M}^\perp, \\ 
\mathcal{H}^C(\mathfrak{M}) & := \{ M \in \Modu (S) \ \text{ : } \ M \simeq C \otimes_R A \text{ for some } A \in \mathcal{H}(\mathfrak{M}) \}.
\end{align*}

\begin{definition}
We say that an $S$-module $G$ is \textbf{weakly $\bm{(\mathcal{H}^C(\mathfrak{M}),\mathcal{H}_C(\mathfrak{N}))}$-Gorenstein flat} if there exists a $\mathcal{H}^C(\mathfrak{M})$-coresolution of $G$, say 
\[
0 \to G \to H^0 \xrightarrow{\partial^0} H^1 \xrightarrow{\partial^1} \cdots 
\]
such that $\Ker(\partial^k) \in [\mathcal{H}_C(\mathfrak{N})]^\top$ for every $k \in \mathbb{Z}_{\geq 0}$. 
\textbf{Weakly $\bm{(\mathcal{M}^C(S),\mathcal{H}_C(\mathfrak{N}))}$-Gorenstein flat} $S$-modules are defined similarly, by considering $\mathcal{M}^C(S)$-coresolutions of $G$ with cycles in $[\mathcal{H}_C(\mathfrak{N})]^\top$. 
\end{definition}

Let $\mathcal{WGF}_{(\mathcal{H}^C(\mathfrak{M}),\mathcal{H}_C(\mathfrak{N}))}$ and $\mathcal{WGF}_{(\mathcal{M}^C(S),\mathcal{H}_C(\mathfrak{N}))}$ denote the classes of weakly $(\mathcal{H}^C(\mathfrak{M}),\mathcal{H}_C(\mathfrak{N}))$-Gorenstein flat and weakly $(\mathcal{M}^C(S),\mathcal{H}_C(\mathfrak{N}))$-Gorenstein flat $S$-modules. Before proving characterizations and properties for these classes, let us summarize the properties of $\mathcal{H}^C(\mathfrak{M})$ is the following result and provide an additional property for $\mathcal{H}_C(\mathfrak{N})$.

\begin{lemma} \label{raro2}
The following assertions hold:
\begin{enumerate}
\item The following are equivalent for every $M \in \Modu (S)$:
\begin{enumerate}[(a)]
\item $M \in \mathcal{M}^C(S) \cap [\mathcal{M}^C(S)]^\perp$.
\item $M \in \mathcal{H}^C(\mathfrak{M})$.
\item $\Hom_S(C,M) \in \mathcal{H}(\mathfrak{M})$.
\end{enumerate}

\item There is an equivalence of categories 
\[
\begin{tikzpicture}[description/.style={fill=white,inner sep=2pt}] 
\matrix (m) [ampersand replacement=\&, matrix of math nodes, row sep=5em, column sep=6em, text height=1.25ex, text depth=0.25ex] 
{ 
\mathcal{H}(\mathfrak{M}) \& \mathcal{H}^C(\mathfrak{M}) \\
}; 
\path[->] 
(m-1-1) edge [bend left = 20] node[above] {\footnotesize$C \otimes_R \sim$} node[below] {\footnotesize$\sim$} (m-1-2)
(m-1-2) edge [bend left = 20] node[above] {\footnotesize$\sim$} node[below] {\footnotesize$\Hom_{S}(C,\sim)$} (m-1-1)
;
\end{tikzpicture} 
\]

\item There is an equivalence of categories 
\[
\begin{tikzpicture}[description/.style={fill=white,inner sep=2pt}] 
\matrix (m) [ampersand replacement=\&, matrix of math nodes, row sep=5em, column sep=4.5em, text height=1.25ex, text depth=0.25ex] 
{ 
[\mathcal{H}(\mathfrak{M})]^\perp \cap \A_C(R) \& \B_C(S) \cap [\mathcal{H}^C(\mathfrak{M})]^\perp \\
}; 
\path[->] 
(m-1-1) edge [bend left = 15] node[above] {\footnotesize$C \otimes_R \sim$} node[below] {\footnotesize$\sim$} (m-1-2)
(m-1-2) edge [bend left = 15] node[above] {\footnotesize$\sim$} node[below] {\footnotesize$\Hom_{S}(C,\sim)$} (m-1-1)
;
\end{tikzpicture} 
\]

\item There is an equivalence of categories
\[
\begin{tikzpicture}[description/.style={fill=white,inner sep=2pt}] 
\matrix (m) [ampersand replacement=\&, matrix of math nodes, row sep=3em, column sep=6em, text height=1.25ex, text depth=0.25ex] 
{ 
[\mathcal{H}(\mathfrak{N})]^\top \cap \mathcal{A}_C(R) \& \mathcal{B}_C(S) \cap [\mathcal{H}_C(\mathfrak{N})]^\top \\
}; 
\path[->] 
(m-1-1) edge [bend left = 15] node[above] {\footnotesize$C \otimes_R \sim$} node[below] {\footnotesize$\sim$} (m-1-2)
(m-1-2) edge [bend left = 15] node[above] {\footnotesize$\sim$} node[below] {\footnotesize$\Hom_S(C,\sim)$} (m-1-1)
;
\end{tikzpicture} 
\]
\end{enumerate}
\end{lemma}

\begin{proof} ~\\
\begin{enumerate}
\item For (a) $\Rightarrow$ (b), let $M \in \mathcal{M}^C(S) \cap [\mathcal{M}^C(S)]^\perp$. Then, in particular, $M \simeq C \otimes_R A$ for some $A \in \mathfrak{M}$. We show that $A \in \mathfrak{M}^\perp$. So let $A' \in \mathfrak{M}$, then by Lemma \ref{lem:HolmWhite_for_MN} (1) we get
\[
\Ext^i_R(A',A) \simeq \Ext^i_S(C \otimes_R A', C \otimes_R A) \simeq \Ext^i_S(C \otimes_R A', M) = 0.
\]
Thus, $A \in \mathfrak{M}^\perp$ and so $M \simeq C \otimes_R A$ with $A \in \mathcal{H}(\mathfrak{M})$. 

The implication (b) $\Rightarrow$ (c) is an immediate consequence of the first Foxby equivalence. 

Finally, for (c) $\Rightarrow$ (a), suppose $\Hom_S(C,M) \in \mathcal{H}(\mathfrak{M})$. Then, $M \simeq C \otimes_R \Hom_S(C,M) \in \mathcal{M}^C(S)$ by Proposition \ref{Equivalencia1}. Now let $C \otimes_R A' \in \mathcal{M}^C(S)$ (that is, $A' \in \mathfrak{M}$). Again, by Lemma \ref{lem:HolmWhite_for_MN} (1) we get
\[
\Ext^i_S(C \otimes_R A',M) \simeq \Ext^i_R(A',\Hom_S(C,M)) = 0.
\]
Then, $M \in [\mathcal{M}^C(S)]^\perp$ as well. 

\item It is an immediate consequence of part (1). 

\item Let $A \in [\mathcal{H}(\mathfrak{M})]^\perp \cap \mathcal{A}_C(R)$ and consider $H \in \mathcal{H}^C(\mathfrak{M})$, that is, $H \simeq C \otimes_R A'$ with $A' \in \mathfrak{M}$. By \cite[Thm. 6.4 (a)]{HolmWhite07}, we have
\[
\Ext^{i}_S(H, C\otimes_R A) \simeq \Ext^{i}_S(C\otimes_R A', C\otimes_R A) \simeq \Ext^{i}_R(A',A) = 0.
\]
Hence, $C\otimes_R A \in \B_C(S) \cap [\mathcal{H}^C(\mathfrak{M})]^\perp$. The fact that $\Hom_S(C,\sim)$ maps $S$-modules in $\B_C(S) \cap [\mathcal{H}^C(\mathfrak{M})]^\perp$ into $R$-modules in $[\mathcal{H}(\mathfrak{M})]^\perp \cap \mathcal{A}_C(R)$ can be proved similarly. 

\item Let $A \in [\mathcal{H}(\mathfrak{N})]^\top \cap \mathcal{A}_C(R)$ and consider $H \in \mathcal{H}_C(\mathfrak{N})$, that is, $H \simeq \Hom_{R^{\rm op}}(C,B)$ with $B \in \mathcal{H}(\mathfrak{N})$. By Lemma \ref{lem:HolmWhite_for_MN} (3), we have
\begin{align*}
\Tor^{i}_S(H, C\otimes_R A) & \simeq \Tor^{i}_S(\Hom_{R^{\rm op}}(C,B), C \otimes_R A) \simeq \Tor^{i}_{R}(B,A) = 0.
\end{align*}
Hence, $C\otimes_R A \in \mathcal{B}_C(S) \cap [\mathcal{H}_C(\mathfrak{N})]^\top$. The fact that $\Hom_S(C,\sim)$ maps $S$-modules in $\mathcal{B}_C(S) \cap [\mathcal{H}_C(\mathfrak{N})]^\top$ into $R$-modules in $[\mathcal{H}(\mathfrak{N})]^\top \cap \mathcal{A}_C(R)$ can be proved similarly. 
\end{enumerate}
\end{proof}

The following result shows how $S$-modules in $\mathcal{WGF}_{(\mathcal{H}^C(\mathfrak{M}),\mathcal{H}_C(\mathfrak{N}))}$ behave under the first Foxby equivalence. 

\begin{proposition}
The following assertions hold for every $N \in \mathcal{B}_C(S)$ and $M \in \mathcal{A}_C(R)$:
\begin{enumerate}
\item $N \in \mathcal{WGF}_{(\mathcal{H}^C(\mathfrak{M}),\mathcal{H}_C(\mathfrak{N}))}$ if, and only if, $\Hom_S(C,N) \in \mathcal{WGF}_{(\mathcal{H}(\mathfrak{M}),\mathcal{H}(\mathfrak{N}))}$.

\item $M \in \mathcal{WGF}_{(\mathcal{H}(\mathfrak{M}),\mathcal{H}(\mathfrak{N}))}$ if, and only if, $C \otimes_R M \in \mathcal{WGF}_{(\mathcal{H}^C(\mathfrak{M}),\mathcal{H}_C(\mathfrak{N}))}$.
\end{enumerate}
\end{proposition}

\begin{proof}
For part (1), if $N \in \mathcal{WGF}_{(\mathcal{H}^C(\mathfrak{M}),\mathcal{H}_C(\mathfrak{N}))}$  there exists a $\mathcal{H}^C(\mathfrak{M})$-coresolution
\[
\varepsilon \colon 0 \to N \to H^0 \xrightarrow{\partial^0} H^1 \xrightarrow{\partial^1} \cdots 
\]
of $N$ such that $\Ker(\partial^k) \in [\mathcal{H}_C(\mathfrak{N})]^\top$ for every $k \in \mathbb{Z}_{\geq 0}$. Since $N \in \mathcal{B}_C(S)$, $\mathcal{H}^C(\mathfrak{M}) \subseteq \mathcal{B}_C(S)$ and $\mathcal{B}_C(S)$ is coresolving, $\varepsilon$ has cycles in $\mathcal{B}_C(S)$, and so 
\[
\Hom_S(C,\varepsilon) \colon 0 \to \Hom_S(C,N) \to \Hom_S(C,H^0) \xrightarrow{\partial^0} \Hom_S(C,H^1) \xrightarrow{\partial^1} \cdots 
\]
is exact. By Lemma \ref{raro2} (2), $\Hom_S(C,\varepsilon)$ is a $\mathcal{H}(\mathfrak{M})$-coresolution of $\Hom_S(C,N)$. Moreover, by By Lemma \ref{raro2} (4) we have that $\Hom_S(C,\varepsilon)$ has cycles in $[\mathcal{H}(\mathfrak{N})]^\top$. Then, $\Hom_S(C,N) \in \mathcal{WGF}_{(\mathcal{H}(\mathfrak{M}),\mathcal{H}(\mathfrak{N}))}$. 

The ``if'' part can be proved similarly, while (2) is a consequence of (1).
\end{proof}

Modules in $\mathcal{WGF}_{(\mathcal{H}^C(\mathfrak{M}),\mathcal{H}_C(\mathfrak{N}))}$ can be characterized as other types of relative Gorenstein flat modules, under certain conditions, as we show below in Proposition \ref{prop_coCMflat}. We shall need the following property.

\begin{lemma}\label{lemaWGFext}
The following assertions hold.
\begin{enumerate} 
\item $\Tor_{\geq 1}^S(\mathcal{H}_C(\mathfrak{N}), \mathcal{M}^C(S)) = 0$. 

\item If $0 \to K \to Q \to M \to 0$ is a short exact sequence with $K \in \mathcal{WGF}_{(\mathcal{M}^C(S),\mathcal{H}_C(\mathfrak{N}))}$ and $M \in \mathcal{M}^C(S)$, then $Q \in \mathcal{WGF}_{(\mathcal{M}^C(S),\mathcal{H}_C(\mathfrak{N}))}$ provided that $\mathcal{M}$ is closed under extensions.
\end{enumerate}
\end{lemma}

\begin{proof}
For part (1), let $M \in \mathcal{M}^C(S)$ and $H \in \mathcal{H}_C(\mathfrak{N})$, that is, $M \simeq C \otimes_R A$ with $A \in \mathfrak{M}$ and $H \simeq \Hom_{R^{\rm op}}(C,B)$ with $B \in \mathcal{H}(\mathfrak{N})$. By Lemma \ref{lem:HolmWhite_for_MN} (3), we have
\[
\Tor^{i}_S(H, M) \simeq \Tor^{i}_S(\Hom_{R^{\rm op}}(C,B), C \otimes_R A) \simeq \Tor^{i}_{R^{\rm op}}(A,B).
\]
Now by \cite[Thm. 3.2.5]{EnochsJenda00}, we have $[\Tor^{i}_{R^{\rm op}}(A,B)]^+ \simeq \Ext^{i}_{R^{\rm op}}(B,A^+) = 0$, where the last equality is a consequence of $B \in {}^\perp\mathfrak{N}$ and $A \in \mathfrak{M}$ (and so $A^+ \in \mathfrak{N}$). It then follows that $\Tor^{i}_S(H, M) = 0$.

For part (2), since $K \in [\mathcal{H}_C(\mathfrak{N})]^\top$ by the definition of $\mathcal{WGF}_{(\mathcal{M}^C(S),\mathcal{H}_C(\mathfrak{N}))}$, $M \in [\mathcal{H}_C(\mathfrak{N})]^\top$ by part (1), and $[\mathcal{H}_C(\mathfrak{N})]^\top$ is closed under extensions, we have that $Q \in [\mathcal{H}_C(\mathfrak{N})]^\top$. Now consider a $\mathcal{M}^C(S)$-coresolution
\[
0 \to K \to M^0 \xrightarrow{\partial^0} M^1 \xrightarrow{\partial^1} \cdots 
\]
of $K$ such that $\Ker(\partial^k) \in [\mathcal{H}_C(\mathfrak{N})]^\top$ for every $k \in \mathbb{Z}_{\geq 0}$. Taking the pushout of $M^0 \leftarrow K \to Q$ yields the following commutative diagram with exact rows and columns:
\[
\begin{tikzpicture}[description/.style={fill=white,inner sep=2pt}] 
\matrix (m) [matrix of math nodes, row sep=2.3em, column sep=2.3em, text height=1.25ex, text depth=0.25ex] 
{ 
{} & 0 & 0 & {} & {} \\
0 & K & Q & M & 0 \\
0 & M^0 & M^0_Q & M & 0 \\
{} & M^1 &M^1 & {} & {} \\
{} & M^2 &M^2 & {} & {} \\
{} & \vdots & \vdots & {} & . \\
}; 
\path[->] 
(m-2-2)-- node[pos=0.5] {\footnotesize$\mbox{\bf po}$} (m-3-3) 
(m-1-2) edge (m-2-2) (m-1-3) edge (m-2-3)
(m-2-1) edge (m-2-2) (m-2-2) edge (m-2-3) edge (m-3-2) (m-2-3) edge (m-2-4) edge (m-3-3) (m-2-4) edge (m-2-5)
(m-3-1) edge (m-3-2) (m-3-2) edge (m-3-3) edge node[left] {\footnotesize$\partial^0$} (m-4-2) (m-3-3) edge (m-3-4) edge node[right] {\footnotesize$\partial^0_Q$} (m-4-3) (m-3-4) edge (m-3-5) 
(m-4-2) edge node[left] {\footnotesize$\partial^1$} (m-5-2) (m-4-3) edge node[right] {\footnotesize$\partial^1$} (m-5-3)
(m-5-2) edge (m-6-2) (m-5-3) edge (m-6-3)
;
\path[-,font=\scriptsize]
(m-2-4) edge [double, thick, double distance=2pt] (m-3-4)
(m-4-2) edge [double, thick, double distance=2pt] (m-4-3)
(m-5-2) edge [double, thick, double distance=2pt] (m-5-3)
;
\end{tikzpicture} 
\]
Note that $0 \to Q \to M^0_Q \xrightarrow{\partial^0_Q} M^1 \xrightarrow{\partial^1} M^2 \to \cdots$ has cycles in $\Ker(\partial^k) \in [\mathcal{H}_C(\mathfrak{N})]^\top$. Moreover, by Proposition \ref{Prop0} (1), $\mathcal{M}^C(S)$ is closed under extensions, and so $M^0_Q \in \mathcal{M}^C(S)$. Hence, the result follows.
\end{proof}

\begin{proposition}\label{prop_coCMflat}
Consider the following assertions for every $S$-module $N$:
\begin{enumerate}[(a)]
\item $N \in \mathcal{WGF}_{(\mathcal{H}^C(\mathfrak{M}),\mathcal{H}_C(\mathfrak{N}))}$. 

\item $N \simeq \Ker(\partial^0)$ in some exact complex 
\[
\Gamma \colon \cdots \to F_1 \to F_0 \to M^0 \xrightarrow{\partial^0} M^1 \to \cdots
\]
with $F_k \in \Flat(S)$ and $M^k \in \mathcal{M}^C(S)$ for every $k \in \mathbb{Z}_{\geq 0}$, such that $H \otimes_S \Gamma$ is exact for every $H \in \mathcal{H}_C(\mathfrak{N})$. 

\item There exists a short exact sequence $0 \to N \to M \to L \to 0$ with $M \in \mathcal{M}^C(S)$ and $L \in \mathcal{WGF}_{(\mathcal{M}^C(S),\mathcal{H}_C(\mathfrak{N}))}$.
\end{enumerate}
Then the implications (a) $\Rightarrow$ (c) $\Rightarrow$ (b) hold. Furthermore, (b) $\Rightarrow$ (a) holds if $(\mathcal{M,N})$ is perfect and $\mathcal{M}$ is resolving. 
\end{proposition}

\begin{proof}
The implication (a) $\Rightarrow$ (c) is immediate from the definition of $\mathcal{WGF}_{(\mathcal{H}^C(\mathfrak{M}),\mathcal{H}_C(\mathfrak{N}))}$ and the containment $\mathcal{H}^C(\mathfrak{M}) \subseteq \mathcal{M}^C(S)$ (by Lemma \ref{raro2} (1)). 

For (c) $\Rightarrow$ (b), consider a $\mathcal{M}^C(S)$-coresolution $0 \to L \to M^0 \xrightarrow{\partial^0} M^1 \xrightarrow{\partial^1} \cdots$ with cycles in $[\mathcal{H}_C(\mathfrak{N})]^\top$ (in particular, $L \in [\mathcal{H}_C(\mathfrak{N})]^\top$). Glueing at $L$ this coresolution with the given short exact sequence yields a $\mathcal{M}^C(S)$-coresolution of $N$, namely $0 \to N \to M \to M^0 \xrightarrow{\partial^0} M^1 \xrightarrow{\partial^1} M^2 \to \cdots$. Let us show that $N \in [\mathcal{H}_C(\mathfrak{N})]^\top$. Note that $M \in [\mathcal{H}_C(\mathfrak{N})]^\top$ by Lemma \ref{lemaWGFext} (1). Thus, we have a short exact sequence $0 \to N \to M \to L \to 0$ with $M \in [\mathcal{H}_C(\mathfrak{N})]^\top$ and $L \in [\mathcal{H}_C(\mathfrak{N})]^\top$, and since $[\mathcal{H}_C(\mathfrak{N})]^\top$ is resolving we have that $N \in [\mathcal{H}_C(\mathfrak{N})]^\top$, and hence the previous $\mathcal{M}^C(S)$-coresolution of $N$ has cycles in $[\mathcal{H}_C(\mathfrak{N})]^\top$. On the other hand, consider a flat resolution of $N$, say $\cdots \to F_1 \to F_0 \to N \to 0$. Applying the same argument, this flat resolution has cycles in $[\mathcal{H}_C(\mathfrak{N})]^\top$, and hence we obtain a complex
\[
\Gamma \colon \cdots \to F_1 \to F_0 \to M \to M^0 \xrightarrow{\partial^0} M^1 \xrightarrow{\partial^1} M^2 \to \cdots
\]
with cycles in $[\mathcal{H}_C(\mathfrak{N})]^\top$ and where $N = \Ker(M \to M^0)$, which in turn implies that $H \otimes_S \Gamma$ is exact for every $H \in \mathcal{H}_C(\mathfrak{N})$.  

Finally, for (b) $\Rightarrow$ (a), suppose $N \simeq \Ker(\partial^0)$ in some exact complex 
\[
\Gamma \colon \cdots \to F_1 \to F_0 \to M^0 \xrightarrow{\partial^0} M^1 \to \cdots
\]
with $F_k \in \Flat(S)$ and $M^k \in \mathcal{M}^C(S)$ for every $k \in \mathbb{Z}_{\geq 0}$, such that $H \otimes_S \Gamma$ is exact for every $H \in \mathcal{H}_C(\mathfrak{N})$. Since $\Gamma$ has components in $[\mathcal{H}_C(\mathfrak{N})]^\top$, the latter acyclicity condition implies that $\Gamma$ has cycles in $[\mathcal{H}_C(\mathfrak{N})]^\top$ (again by Lemma \ref{lemaWGFext} (1)). Thus, $0 \to N \to M^0 \xrightarrow{\partial^0} M^1 \xrightarrow{\partial^1} \cdots$ is a $\mathcal{M}^C(S)$-coresolution of $N$ has cycles in $[\mathcal{H}_C(\mathfrak{N})]^\top$. Consider the short exact sequence $0 \to N \to M^0 \to K \to 0$ where $K = \Ker(\partial^1)$. Write $M^0 \simeq C \otimes_R A$ for some $A \in \mathfrak{M}$. By \cite[Prop. 2.8]{HJ09}, $(\mathfrak{M},\mathfrak{N})$ is a perfect duality pair, and so $(\mathfrak{M},\mathfrak{M}^\perp)$ is a perfect cotorsion pair (in particular complete). On the other hand, $\mathfrak{M}$ is resolving. So there is a short exact sequence $0 \to A \to H \to A' \to 0$ with $H \in \mathfrak{M} \cap \mathfrak{M}^\perp = \mathcal{H}(\mathfrak{M})$ and $A' \in \mathfrak{M}$. Since $A' \in \mathcal{A}_C(R)$, the sequence $0 \to M^0 \to C \otimes_R H \to C \otimes_R A' \to 0$ is exact.  Now taking the pushout of $K \leftarrow M^0 \to C \otimes_R H$ yields the following commutative diagram with exact rows and columns:
\[
\begin{tikzpicture}[description/.style={fill=white,inner sep=2pt}] 
\matrix (m) [matrix of math nodes, row sep=2.3em, column sep=2.3em, text height=1.25ex, text depth=0.25ex] 
{ 
{} & {} & 0 & 0 & {} \\
0 & N & M^0 & K & 0 \\
0 & N & C \otimes_R H & Q & 0 \\
{} & {} & C \otimes_R A' & C \otimes_R A' & {} \\
{} & {} & 0 & 0 & {} \\
}; 
\path[->] 
(m-2-3)-- node[pos=0.5] {\footnotesize$\mbox{\bf po}$} (m-3-4) 
(m-1-3) edge (m-2-3) (m-1-4) edge (m-2-4)
(m-2-1) edge (m-2-2) (m-2-2) edge (m-2-3) (m-2-3) edge (m-2-4) edge (m-3-3) (m-2-4) edge (m-2-5) edge (m-3-4)
(m-3-1) edge (m-3-2) (m-3-2) edge (m-3-3) (m-3-3) edge (m-3-4) edge (m-4-3) (m-3-4) edge (m-3-5) edge (m-4-4)
(m-4-3) edge (m-5-3) (m-4-4) edge (m-5-4)
;
\path[-,font=\scriptsize]
(m-2-2) edge [double, thick, double distance=2pt] (m-3-2)
(m-4-3) edge [double, thick, double distance=2pt] (m-4-4)
;
\end{tikzpicture} 
\]
Since $K \in \mathcal{WGF}_{(\mathcal{M}^C(S),\mathcal{H}_C(\mathfrak{N}))}$ and $C \otimes_R A' \in \mathcal{M}^C(S)$, then $Q \in \mathcal{WGF}_{(\mathcal{M}^C(S),\mathcal{H}_C(\mathfrak{N}))}$, and hence we get the short exact sequence 
\[
0 \to N \to C \otimes_R H \to Q \to 0
\]
with $C \otimes_R H \in \mathcal{H}^C(\mathfrak{M})$ and $Q \in [\mathcal{H}^C(\mathfrak{M})]^\perp$. Iterating the previous procedure yields a $\mathcal{H}^C(\mathfrak{M})$-coresolution of $N$ with cycles in $[\mathcal{H}^C(\mathfrak{M})]^\perp$, and therefore (a) holds. 
\end{proof}

\begin{proposition}\label{WGF_meets_GF}
If $(\mathcal{M,N})$ is a perfect duality pair with $\mathcal{M}$ resolving and $S \in \mathcal{M}^C(S)$, 
\[
\mathcal{WGF}_{(\mathcal{H}^C(\mathfrak{M}),\mathcal{H}_C(\mathfrak{N}))} = \mathcal{GF}_{(\mathcal{M}^C(S),\mathcal{H}_C(\mathfrak{N}))}.
\]
\end{proposition}

\begin{proof}
By Proposition \ref{prop_coCMflat} (a) $\Rightarrow$ (b), for every $N \in \mathcal{WGF}_{(\mathcal{H}^C(\mathfrak{M}),\mathcal{H}_C(\mathfrak{N}))}$ one has that $N \simeq \Ker(\partial^0)$ in some exact complex 
\[
\Gamma \colon \cdots \to F_1 \to F_0 \to M^0 \xrightarrow{\partial^0} M^1 \to \cdots
\]
with $F_k \in \Flat(S)$ and $M^k \in \mathcal{M}^C(S)$ for every $k \in \mathbb{Z}_{\geq 0}$, such that $H \otimes_S \Gamma$ is exact for every $H \in \mathcal{H}_C(\mathfrak{N})$. By Proposition \ref{Prop0} (4), note that $F_k \in \Flat(S) \subseteq \mathcal{M}^C(S)$. It then follows that $N \in \mathcal{GF}_{(\mathcal{M}^C(S),\mathcal{H}_C(\mathfrak{N}))}$. By Lemma \ref{lemaWGFext} (1) the latter implies that $\Gamma$ has cycles in $[\mathcal{H}_C(\mathfrak{N})]^\top$, and so $N \in \mathcal{WGF}_{(\mathcal{M}^C(S),\mathcal{H}_C(\mathfrak{N}))}$. Hence, $N \in \mathcal{WGF}_{(\mathcal{H}^C(\mathfrak{M}),\mathcal{H}_C(\mathfrak{N}))}$ by Proposition \ref{prop_coCMflat}. 
\end{proof}

\begin{remark}
If ${}_S C_R$ is faithfully semidualizing, $(\mathcal{M,N})$ is a bicomplete duality pair with $\mathcal{M}$ resolving and closed under products, $S \in \mathcal{M}^C(S)$ and $({}^\perp[\B_C(R^{\rm op})],\B_C(R^{\rm op}))$ is a cotorsion pair cogenerated by a set, then by Propositions \ref{coCM_meets_GI} and \ref{WGF_meets_GF}, we have that
\[
\mathrm{coCM}(\mathcal{H}_C(\mathfrak{N})) = \mathcal{GI}_{(\mathcal{H}_C(\mathfrak{N}),\mathcal{N}^C(S^{\rm op}))} \ \ \text{and} \ \ \mathcal{WGF}_{(\mathcal{H}^C(\mathfrak{M}),\mathcal{H}_C(\mathfrak{N}))} = \mathcal{GF}_{(\mathcal{M}^C(S),\mathcal{H}_C(\mathfrak{N}))}.
\] 
Thus, the pair $(\mathcal{WGF}_{(\mathcal{H}^C(\mathfrak{M}),\mathcal{H}_C(\mathfrak{N}))},\mathrm{coCM}(\mathcal{H}_C(\mathfrak{N})))$ has the homological properties described in Corollary \ref{relationGFvsGI}. 
\end{remark}


\section*{Funding}

The first author was fully supported by a SECIHTI posdoctoral fellowship at the Centro de Ciencias Matem\'aticas - UNAM, and partially supported by PAPIIT-UNAM project \# IN100124. The second author was supported by ANII - Agencia Nacional de Investigaci\'on e Innovaci\'on, and PEDECIBA - Programa de Desarrollo de las Ciencias B\'asicas. 

Part of this research was carried out while the first author was visiting the Instituto de Matem\'atica y Estad\'istica ``Prof. Ing. Rafael Laguardia'' at the Universidad de la Rep\'ublica, in August of 2024 and July of 2025, and supported with funds from Programa de Despegue Cient\'ifico 2023 (PEDECIBA) and PAPIIT-UNAM IN100124. He wants to thank IMERL-UdelaR faculty and staff for their hospitality and kindness.


\bibliographystyle{plain}
\bibliography{biblio_semidualizante}




\end{document}